\documentclass[letter, 10pt, conference]{ieeeconf}

\IEEEoverridecommandlockouts

\usepackage{amsmath,amssymb,pstricks,color,dsfont}
\usepackage{graphicx}
\usepackage{psfrag,graphicx,epsfig}
\usepackage{xspace}
\usepackage{subfig}
\usepackage{float}
\usepackage{placeins}

\usepackage[vlined,linesnumbered,boxruled]{algorithm2e}

\usepackage{algorithm2e}

\newcommand{\nnum}{\nonumber}

\newcommand{\EQ}{\begin{eqnarray}}
\newcommand{\EN}{\end{eqnarray}}
\newcommand{\EQQ}{\begin{eqnarray*}}
\newcommand{\ENN}{\end{eqnarray*}}

\newcommand{\bremark}{\begin{remark} \begin{rm} }
\newcommand{\eremark}{ \end{rm} \rule{1mm}{2mm}
\end{remark} }
\newcommand{\btheorem}{\begin{theorem} \begin{rm} }
\newcommand{\etheorem}{ \end{rm} \rule{1mm}{2mm}
\end{theorem} }
\newcommand{\blemma}{\begin{lemma} \begin{rm} }
\newcommand{\elemma}{ \end{rm} \rule{1mm}{2mm}
\end{lemma} }
\newcommand{\bcorollary}{\begin{corollary} \begin{rm} }
\newcommand{\ecorollary}{ \end{rm} \rule{1mm}{2mm}
\end{corollary} }
\newcommand{\bdefinition}{\begin{definition}\begin{rm} }
\newcommand{\edefinition}{ \end{rm} \rule{1mm}{2mm}
\end{definition} }
\newcommand{\bproposition}{\begin{proposition} \begin{rm} }
\newcommand{\eproposition}{ \end{rm} \rule{1mm}{2mm}
\end{proposition} }
\newcommand{\bexample}{\begin{example} \begin{rm} }
\newcommand{\eexample}{ \end{rm} \rule{1mm}{2mm}
\end{example} }
\newcommand{\basm}{\begin{assumption} \begin{rm}}
\newcommand{\easm}{\end{rm} %\rule{1mm}{2mm}
\end{assumption}}

\newcommand{\real}{\mathds{R}}

\newcommand{\iGame}{\text{iGame}}
\newcommand{\iGameStar}{\text{iGame$^*$}}

\newcommand{\EE}{\mathcal{E}}
\newcommand{\FF}{\mathcal{F}}
\newcommand{\GG}{\mathcal{G}}

\newtheorem{theorem}{\bf Theorem}[section]
\newtheorem{lemma}{\bf Lemma}[section]
\newtheorem{definition}{\bf Definition}[section]
\newtheorem{remark}{\bf Remark}[section]
\newtheorem{corollary}{\bf Corollary}[section]
\newtheorem{proposition}{\bf Proposition}[section]
\newtheorem{example}{\bf Example}[section]
\newtheorem{assumption}{\bf Assumption}[section]

\newcommand\oprocendsymbol{\hbox{$\bullet$}}
\newcommand\oprocend{\relax\ifmmode\else\unskip\hfill\fi\oprocendsymbol}

\date{}

\begin{document}

\title{Anytime computation algorithms for approach-evasion differential games}

\author{Erich Mueller$^{1,2}$, Minghui Zhu$^{1,3}$, Sertac Karaman$^2$, Emilio Frazzoli$^2$
\thanks{$^1$Equal contribution from both authors.}
\thanks{$^2$The authors are with Laboratory for Information and Decision Systems, Massachusetts
Institute of Technology, 77 Massachusetts Avenue, Cambridge, MA 02139, {\tt\small
\{emueller,sertac,frazzoli\}@mit.edu}.}
\thanks{$^3$M. Zhu is with the Department of Electrical Engineering, Pennsylvania State University, 201 Old Main, University Park, PA 16802, {\tt muz16@psu.edu}.}
\thanks{This research was supported in part by ONR Grant \#N00014-09-1-0751 and the Michigan/AFRL Collaborative Center on Control Sciences, AFOSR grant \#FA 8650-07-2-3744.}}

\maketitle

\begin{abstract} This paper studies a class of approach-evasion differential games, in which one player aims to steer the state of a dynamic system to the given target set in minimum time, while avoiding some set of disallowed states, and the other player desires to achieve the opposite. We propose a class of novel anytime computation algorithms, analyze their convergence properties and verify their performance via a number of numerical simulations. Our algorithms significantly outperform the multi-grid method for the approach-evasion differential games both theoretically and numerically. Our technical approach leverages incremental sampling in robotic motion planning and viability theory.\end{abstract}

\section{Introduction}

Recent advances in embedded computation and communication have boosted the emergence of cyber-physical systems. Cyber-physical systems are characterized by the strong coupling between the cyber space and the physical world. A number of cyber-physical systems; e.g., autonomous vehicles and the power grid, operate with time-varying computing resources in dynamically changing and uncertain operating environments. This feature demands that the controllers are synthesized in an \emph{anytime} and online fashion where a feasible controller is quickly returned and its optimality is continuously improved as more processing time is available.

A fundamental problem of cyber-physical systems is safety control; i.e., controlling dynamic systems to achieve the given objectives and simultaneously stay in the given safety sets. A possible formulation is the approach-evasion differential games where one player, say the angle, desires to steer the system states to the goal set as soon as possible and maintain the system states inside the safety set and the objective of the other player, say the demon, is completely opposite. There are mainly two classes of
numerical schemes for the approach-evasion differential games: one is based on viscosity solutions of partial differential
equations (PDEs); e.g., in~\cite{MB-ICD:97,MB-MF-PS:99,PES:99}, and the other is built on viability theory; e.g.,
in~\cite{JPA:09,JPA-AB-PSP:11,PC-MQ-PSP:99}. Both methods are built on the \emph{multi-grid} successive approximation (or coarse-to-refined) approach in; i.e.,~\cite{CSC-JNT:91} and the asymptotic optimality is provable. On the other hand, one can always replace the multi-grid method by the \emph{fixed-grid} method in the above papers where a grid with fixed resolution is chosen before implementing the algorithms. For this method, the approximation error is determined by the fixed gird resolution, and their relation could be very challenging to characterize, especially for the constrained nonlinear system of interest. In addition, a single predetermined gird is not capable of dealing with different scenarios in dynamically changing environments.

In the robotics society, a relevant problem of robotic motion planning has been receiving considerable attention. In robotic motion planning, a feasible, collision-free and open-loop path through a cluttered environment is found, connecting an initial configuration or state to a target region. It is well-known that the problem is computationally challenging~\cite{JHR:79}. The Rapidly-exploring Random Tree (RRT) algorithm and its variants; e.g.,
in~\cite{SML-JJK:00,SML-JJK:01}, are able to find a feasible path quickly. Recently, two novel algorithms, PRM$^*$ and
RRT$^*$, have been developed in~\cite{SK-EF:11}, and shown to be computationally efficient and asymptotically optimal.
The methodology of incremental sampling in~\cite{SK-EF:11,SML-JJK:00,SML-JJK:01} has been extended to stochastic optimal control
in~\cite{VH-SK-EF:12} and stochastic filtering in~\cite{PC-SK-EF:12}.

\emph{Contributions.} The paper proposes a class of novel anytime computation algorithms to solve the approach-evasion differential games, and our algorithms significantly outperform the multi-grid method both theoretically and numerically. In particular, we first propose the $\iGame$ Algorithm by leveraging incremental sampling and viability theory in~\cite{PC-MQ-PSP:99}. On the basis of random samples, a sequence of games discrete in the time, states and inputs are constructed to approximate the original game continuous in the time, states and inputs. At each iteration, the values of the state samples are updated \emph{only once} in an \emph{asynchronous} way. The asymptotic convergence of the $\iGame$ Algorithm is formally ensured for all the initial states in contrast to a single initial state in~\cite{PC-MQ-PSP:99}. We refer the readers to Section~\ref{sec:iSaddle-discussion} for a further detailed theoretic comparison with~\cite{PC-MQ-PSP:99}.

The new feature asynchronism of $\iGame$ offers a greater freedom for the online section of sample updates, and opens up a number of opportunities to improve the computational efficiency. We then propose the variation of the $\iGameStar$ Algorithm which explicitly exploits the asynchronism. In particular, $\iGameStar$ maintains a set of directed trees at each iteration and utilizes a novel \emph{cascade update rule}, only updating a subset of state samples whose child changes its estimate at the previous iteration. $\iGameStar$ maintains the same convergence property as $\iGame$.

Most importantly, through the homicidal-chauffeur differential game, we numerically demonstrate that $\iGame$ is faster than the multi-grid method in~\cite{PC-MQ-PSP:99} and $\iGameStar$ is significantly faster than $\iGame$. The anytime property of $\iGameStar$ is also shown in the numerical simulations.

Due to the space limitation, the analysis is omitted in the current paper and provided in the enlarged version~\cite{Mueller.Zhu.Karaman.Frazzoli:13}.

\emph{Literature review.} Among the limited number of differential games for which closed-form solutions have been derived are the homicidal-chauffeur and the lady-in-the-lake games ~\cite{MB-MF-PS:99,Isaacs:99} which are played in unobstructed environments. For more complicated games, numerical methods must be used to determine solutions. The PDE-based and viability-based methods are two main approaches to solve the approach-evasion differential games. For the PDE-based method in~\cite{MB-ICD:97,MB-MF-PS:99,PES:99}, the optimal value functions are characterized as the viscosity solutions of Hamilton-Jacobi-Isaacs equations, and existing numerical schemes for PDEs are applied to approximate the viscosity solutions. This method demands the continuity of the optimal value functions. Unfortunately, this assumption is restrictive, and does not hold for many such games of interest. The more recent viability-based approach in~\cite{JPA:09,JPA-AB-PSP:11,PC-MQ-PSP:99} instead characterizes the optimal value functions via discriminating kernels of the Hamiltonian, and further solve the games through approximating discriminating kernels. This method relaxes the continuity assumption on the
optimal value functions.

Another relevant set of references is about the reachability analysis of dynamic systems. In particular, the
papers~\cite{Lygeros:04,Mitchell.Bayen.Tomlin:05} study the reachability problem in the framework of Hamilton-Jacobi
equations. Numerical algorithms based on level-set methods, e.g., in~\cite{Sethian:96}, have been applied to
reachability computation; e.g., in~\cite{Mitchell.Bayen.Tomlin:04}.

\section{Problem formulation}

In this section, we will present the game formulation of interest and introduce a set of notations and primitive procedures. After that, we will summarize some existing results on viability theory.

\subsection{Problem formulation}

Consider a pair of players, say the angel and the demon. These two players control a dynamical system governed by the
following differential equation: \begin{align}\dot{x}(t) = f(x(t), u(t), w(t)), \label{e7}\end{align} where $x(t)\in
\mathcal{X}\subseteq\real^N$ is the state, and $u(t)\in\mathcal{U}$ (resp. $w(t)\in\mathcal{W}$) is the control of the
angel (resp. demon). For system~\eqref{e7}, the sets of admissible control strategies for players are defined as:
\begin{align*}&\mathcal{U} \triangleq \{u(\cdot) \; : \; [0,+\infty)\rightarrow U,\;{\rm measurable}\},\\ &\mathcal{W}
\triangleq \{w(\cdot) \; : \; [0,+\infty) \rightarrow W,\;{\rm measurable}\},\end{align*} where $U\subseteq
\real^{m_a}$ and $W\subseteq \real^{m_d}$. Let $m \triangleq m_a + m_d$. Denote by $\phi(\cdot;x,u,w) \triangleq
\{\phi(t;x,u,w)\}_{t\geq0}$ the solution to system~\eqref{e7} given the initial state $x$ and controls of $u$ and $w$.

Throughout this paper, we make the following assumption:
\begin{assumption} The following properties hold:
\begin{enumerate}
\item[(A1)] The sets $\mathcal{X}$, $U$ and $W$ are compact.
\item[(A2)] The function $f$ is continuous in $(x,u,w)$ and Lipschitz continuous in $x$ for any $(u,w)\in U\times W$.
\item[(A3)] For any pair of $x\in \mathcal{X}$ and $u\in U$, $F(x, u)$ is convex where the set-valued map $F(x,u) \triangleq \cup_{w\in W}f(x,u,w)$.
\end{enumerate}\label{asm1}
\end{assumption}

\begin{remark} A sufficient condition for (A3) in Assumption~\ref{asm1} is that the set of $W$ is convex and the function $f$ is affine with respect to $w$.\oprocend\end{remark}

We now proceed to describe the game of interest where the objectives of the angel and the demon are completely opposite. In particular, the angel desires to steer $x(t)$ to his open goal set $\mathcal{X}_{\rm goal}\subseteq\real^N$ as soon as possible and simultaneously keep $x(t)$ inside the closed constraint set $\mathcal{X}_{\rm free}\subseteq\real^N$. In contrast, the demon aims to ensure that the system states never reach $\mathcal{X}_{\rm goal}$ and instead leave $\mathcal{X}_{\rm free}$ quickly. As~\cite{NNK-AIS:87}, we will refer the game to as the time-optimal approach-evasion (TO-AE, for short) differential game\footnote{The papers~\cite{MB-MF-PS:99,PC-MQ-PSP:99,MF:06} refer the game we consider as the pursuit-evasion game. As~\cite{TB-GO:99}, we will refer the pursuit-evasion differential games to as those where the dynamics of two players are decoupled and one player aims to capture the other. The homicidal chauffeur game in Section~\ref{sec:simulation} is an example of the pursuit-evasion game. The games of interest can also be applied to robust optimal control where two players act on a single dynamic system.}. The aforementioned conflicting objectives are formalized as follows. Define by $t(x,u,w)$ the first time when the trajectory $\phi(\cdot;x,u,w)$ hits $\mathcal{X}_{\rm goal}$ while staying in $\mathcal{X}_{\rm free}$ before $t(x,u,w)$. More precisely, given the trajectory $\phi(\cdot;x,u,w)$, the quantity $t(x,u,w)$ is defined as follows: \begin{align*}t(x,u,w) \triangleq &\inf\{t\geq0 \; | \; \phi(t;x,u,w)\in \mathcal{X}_{\rm goal},\nnum\\
&\phi(s;x,u,w)\in \mathcal{X}_{\rm free},\;\forall\; s\in[0,t]\}.\end{align*} If $\phi(\cdot;x,u,w)$ leaves $\mathcal{X}_{\rm free}$ before reaching $\mathcal{X}_{\rm goal}$ or never reaches $\mathcal{X}_{\rm goal}$, then $t(x,u,w) = +\infty$. So the angel aims to minimize the hitting time $t(x,u,w)$, and the demon instead wants to maximize the same cost functional.

To define the value of the TO-AE differential game, we need the notion of nonanticipating or causal strategy in the
sense of Varaiya, Roxin, Elloitt and Kalton in~\cite{PV-RJE-ER-NJK:72}. The set $\Gamma^a$ of such strategies for the
angel is such that $\gamma^a\;:\;\mathcal{W}\rightarrow\mathcal{U}$ satisfies for any $T\geq0$, $\gamma^a(w(t)) =
\gamma^a(w'(t))$ for $t\in[0,T]$ if $w(t) = w'(t)$ for $t\in[0,T]$. The lower value of the TO-AE differential game is
given by: \begin{align}T^*(x) = \inf_{\gamma^a(\cdot)\in\Gamma^a}\sup_{w(\cdot)\in
\mathcal{W}}t(x,\gamma^a(w(\cdot)),w(\cdot)).\nnum\end{align} The function $T^*$ will be referred to as the minimum
time function. Note that $t(x,u,w)$ could be infinity and this may cause numerical issues. To deal with this, we
normalize the hitting time by the Kru\v{z}kov transform of $\Psi(r) = 1-e^{-r}$. With this nonlinear transform, we
further define the discounted cost functional $J(x,u,w) = \Psi\circ t(x,u,w)$, and the discounted lower value $v^*$ as
follows: \begin{align}v^*(x) = \inf_{\gamma^a(\cdot)\in\Gamma^a}\sup_{w(\cdot)\in
\mathcal{W}}J(x,\gamma^a(w(\cdot)),w(\cdot)).\nnum\end{align} It is easy to see that $v^*(x) = \Psi\circ T^*(x)$ for
$\forall x \in \mathcal{X}$. We will refer $v^*$ to as the optimal value function.

The objective of the paper is to design and analyze anytime algorithms to approximate $v^*$(or equivalently, $T^*$) and
further synthesize feedback control policies for the angle.

\subsection{Notations and primitive procedures}

Because of (A1) and (A2) in Assumption~\ref{asm1}, $M \triangleq \sup_{x\in \mathcal{X},\; u\in U,\; w\in W}f(x,u,w)$ is well-defined. Let $\ell$ be the Lipschitz constant of $f$ with respect to $x$ for any $(u,w)\in U\times W$. Let $\mu(S)$ be the Lebesgue measure of the set $S$. Let $C_N \triangleq\frac{\pi^{\frac{N}{2}}}{(\frac{N}{2})!}$ be the volume of the unit ball in $\real^N$. Let $\mathcal{B}(x,r)$ be the closed ball centered at $x$ with radius $r$. We will use the short hand $\mathcal{B} = \mathcal{B}(0,1)$ for the unit ball. Define the norms $\|v\|_S \triangleq \sup_{x\in S}\|v(x)\|$ and $\|v - \bar{v}\|_S \triangleq \sup_{x\in S}\|v(x) - \bar{v}(x)\|$ for $v, \bar{v} : S \rightarrow \real$. Denote $\gamma\triangleq \frac{C_N D_s^N}{\mu(\mathcal{X}_{\rm free})}$ where $D_s > 0$ is chosen such that $\gamma > 2$.

%Let $\mathcal{A}(x,T)$ be the set of states in $\mathcal{X}$ which system~\eqref{e7} can reach from $x$ within $T$. Denote $\mathcal{A}(x) \triangleq \bigcup_{T\geq0}\mathcal{A}(x,T)$.

%The following set of primitive procedures will be used in the algorithm design:

The procedure ${\tt Sample}(S,n)$: return $n$ states which are uniformly and independently sampled from the set $S$.

%${\tt Near}(S,y,\alpha)$: return $x\in S$ such that $\|x - y\| \leq \alpha$.

%${\tt NearestNeighbor}(S,z)$: return $x\in S$ such that $\|x - z\| = \inf_{y\in S}\|y - z\|$.
%
%${\tt Proj}(z,S)$: return $x$ which is the projection of $z$ onto the convex set $S$.

\subsection{Preliminary on viability theory}\label{sec:preliminary-viability}

First of all, let us define the set-valued map $\Phi : \mathcal{X} \times \real \times U \rightarrow \mathcal{X}\times \real$ by: for
$x\notin \mathcal{X}_{\rm goal}$, $\Phi(x,\varpi,u) = F(x,u)\times\{-1\}$; otherwise, $\Phi(x,\varpi,u) =
\overline{Co}[\{F(x,u)\times\{-1\}\}\cup\{(0_n,0)\}]$ where $F(x,u) \triangleq \cup_{w\in W}f(x,u,w)$ and $\overline{Co}[A]$ is the convex hull of set $A$. The map $\Phi$ is
the set-valued version of~\eqref{e7} by taking into account all the controls of the demon. In addition, $\Phi$ is
augmented by the scalar variable $\varpi$ representing the time instant. If the state $x$ does not reach $\mathcal{X}_{\rm
goal}$, then the time $\varpi$ decreases at a rate $1$; otherwise, it would stop. The combination of Theorem 5.2 and Lemma
5.4 in~\cite{PC-MQ-PSP:99} shows that the epigraph of $T^*$ is identical to the viability domain for the set-valued map
$\Phi$. That is, if $\Phi$ starts from the initial state $(x(0),\varpi(0))$ with $T^*(x(0))\leq \varpi(0)$, then there is
$u(\cdot)$ which is able to steer the state $x(t)$ to the goal set $\mathcal{X}_{\rm goal}$ no later than $T^*(x(0))$
and hence $\varpi(T^*(x(0))) - T^*(x(0)) \geq 0$. By utilizing this characterization, the authors in~\cite{PC-MQ-PSP:99}
propose a multi-grid successive approximation algorithm on Page 232 to asymptotically reconstruct $T^*$ via a sequence
of fully discretized dynamics of $\Phi$. The readers are referred to~\cite{JPA:09,JPA-AB-PSP:11,PC-MQ-PSP:99} for a
detailed discussion on viability theory for optimal control and differential games.

\section{The $\iGame$ Algorithm}

In this section, we will develop the $\iGame$ Algorithm which integrates incremental sampling
in~\cite{SK-EF:11,SML-JJK:00,SML-JJK:01} and viability theory in~\cite{PC-MQ-PSP:99}. We will also characterize its asymptotic convergence
properties. The main notations used in this section are summarized in Table~\ref{table-alg1}.

\begin{table}[ht]
\caption{Basic notations}
\label{ta:basic}
\centering
\begin{tabular}{|c|l|}
\hline
$y_n \in \mathcal{X}$ & $
\begin{array}{l}
\text{the new sample obtained at iteration $n$}
\end{array}
$
\tabularnewline
\hline
$S_n \subset \mathcal{X}$ & $
\begin{array}{l}
\text{the set of samples obtained up to iteration $n$}
\end{array}
$
\tabularnewline
\hline
$d_n > 0$ & $
\begin{array}{l}
\text{the space discretization size of $S_n$}
\end{array}
$
\tabularnewline
\hline
$h_n > 0$ & $
\begin{array}{l}
\text{the time discretization size at iteration $n$}
\end{array}
$
\tabularnewline
\hline
$\kappa_n \triangleq h_n - d_n$ & $
\begin{array}{l}
\text{the approximate time discretization}
\end{array}
$
\tabularnewline
\hline
$\alpha_n$ & $
\begin{array}{l}
\text{the dilation size at iteration~$n$}
\end{array}
$
\tabularnewline
\hline
$u_n : S_n \rightarrow U$ & $
\begin{array}{l}
\text{the controls for the states in $S_n$}
\end{array}
$
\tabularnewline
\hline
$v_n : S_n \rightarrow [0,1]$ & $
\begin{array}{l}
\text{the discrete value function for the states in $S_n$}
\end{array}
$
\tabularnewline
\hline
$\tilde{v}_n : \mathcal{X} \rightarrow [0,1]$ & $
\begin{array}{l}
\text{the interpolated function of $v_n$}
\end{array}
$
\tabularnewline
\hline
%$\phi : \real^{m_a} \rightarrow \real$ & $
%\begin{array}{l}
%\text{a strictly convex function}
%\end{array}
%$
%\tabularnewline
%\hline
\end{tabular}\label{table-alg1}
\end{table}

\subsection{Algorithm statement}

The $\iGame$ Algorithm starts with an initial sample set $S_0$ and arbitrary initial state $v_0(x)$ in $[0,1]$ for $x \in \mathcal{X}_{\rm goal}\cup \mathcal{X}_{\rm free}$\footnote{The initial state $v_0(x)$ could be any value in $[0,1]$. It is in contrast to $v_0(x)$ must be equal $0$ in~\cite{PC-MQ-PSP:99}. For the simplicity, we generate the initial values by random sampling in the $\iGame$ Algorithm.} and $v_0(x) = 1$ otherwise. At each iteration $n$, a point $y_n$ is uniformly sampled from $\mathcal{X}$, and added into the sample set; i.e., $S_n = S_{n-1}\cup\{y_n\}$. Then the algorithm computes the dispersion $d_n$ such that for any $x\in \mathcal{X}_{\rm free}$, there exists $x'\in S_n$ such that $\|x - x'\| \leq d_n$. The quantity $d_n$ can be viewed as the resolution of the finite grid generated by $S_n$. Based on $d_n$, the algorithm sets the time discretization size $h_n$ such that $d_n = o(h_n)$. Here we choose $h_n = d_n^{\frac{1}{1+\alpha}}$ where $\alpha$ could be any positive scalar. Denote $\kappa_n \triangleq h_n - d_n$ and $R_n$ be an integer lattice on $\real$ consisting of segments of length $d_n$.

With the above set of parameters, we construct the discretization $\Phi_n$ of $\Phi$ as~\cite{PC-MQ-PSP:99}: for $x\in S_n\setminus\mathcal{B}(\mathcal{X}_{\rm goal}, Mh_n+d_n)$, $\Phi_n(x,y,u) = \{x + h_n F(x,u) + \alpha_n\mathcal{B}\}\times\{y-h_n+[-d_n,d_n]\}\cap (S_n\times R_n)$; otherwise, $\Phi_n(x,y,u) = \overline{Co}[\{x + h_n F(x,u) + \alpha_n\mathcal{B}\}\times\{y-h_n+[-d_n,d_n]\cup\{x,y\}\}]\cap (S_n\times R_n)$, where the dilation size $\alpha_n = 2d_n + \ell h_n d_n + M\ell h_n^2$. As shown in~\cite{PC-MQ-PSP:99}, $\Phi_n$ is a good approximation of the continuous counterpart $\Phi$ in Section~\ref{sec:preliminary-viability} and the choice of $\alpha_n$.

We then solve the discretized game associated with $\Phi_n$ for \emph{only one} step. Towards this end, let $v_{n-1} : S_{n-1}\rightarrow[0,1]$ to be the estimate of the value function at iteration $n-1$. After obtaining $y_n$, we need to initialize the estimate, say $\tilde{v}_{n-1}$, on the new grid $S_n$. The point $y_n$ does not bring any new information about $v^*$, and the estimates on $S_{n-1}$ should not be affected by $y_n$. Thus, we choose  $\tilde{v}_{n-1}(x) = v_{n-1}(x)$ for $x\in S_{n-1}$ and $\tilde{v}_{n-1}(y_n) = 1$.

%by noting that $\displaystyle{\max_{x\in \mathcal{B}(y_n,\alpha_n)\cap S_{n-1}}v_{n-1}(x)}$

After initializing $\tilde{v}_{n-1}$, we execute \emph{a single} value iteration on each state in $K_n$, a subset of $S_n$, by performing the discrete operator $F_n$ on $\tilde{v}_{n-1}$ and get the estimate $v_n$. The discrete operator $F_n$ is defined as follows: for $x\in S_n\setminus \mathcal{B}(\mathcal{X}_{\rm goal},M h_n + d_n)$,
\begin{align}&F_n\circ\tilde{v}_{n-1}(x) = 1 - e^{-\kappa_n}\nnum\\
&+ e^{-\kappa_n}\displaystyle{\max_{w\in W}\min_{u\in U_n, y\in \mathcal{B}(x+h_n f(x,u,w),\alpha_n)\cap S_n}}
\tilde{v}_{n-1}(y),\nnum\end{align} and, for $x \in S_n\cap\mathcal{B}(\mathcal{X}_{\rm goal},M h_n + d_n)$, $F_n\circ\tilde{v}_{n-1}(x) = \tilde{v}_{n-1}(x)$. The discrete operator $F_n$ is derived from $(33)$ in~\cite{PC-MQ-PSP:99} by performing the Kru\v{z}kov transform. It will be shown that there is a unique fixed point $v^*_n : S_n \rightarrow [0,1]$ of $\FF_n$ where $v^*_n(x) = \tilde{v}_{n-1}(x)$ for $S_n\cap\mathcal{B}(\mathcal{X}_{\rm goal},M h_n + d_n)$.

The $\iGame$ Algorithm is formally stated in Algorithm~\ref{algorithm:iSaddle}. Here, we would like to mention that due to the Kru\v{z}kov transform, all the computations $v^*$ can be equivalently performed in terms of the estimates of $T^*$. However, the unboundedness of $T^*$ may incur computational issues.

\IncMargin{0.05in}

\begin{algorithm}[h] \small

\For{$x\in S_0 \subset \mathcal{X}$}{$v_0(x) = {\tt Sample}([0,1])$\;}
$n \leftarrow 1$\;

\While{$n < K$}{
$y_n \leftarrow {\tt Sample}(\mathcal{X},1)$\;
$S_n \leftarrow S_{n-1}\cup\{y_n\}$\;
$h_n \leftarrow d_n^{\frac{1}{1+\alpha}}$\;
$\alpha_n \leftarrow 2d_n + \ell h_n d_n + M\ell h_n^2$\;

\For{$x\in S_{n-1}$}{$\tilde{v}_{n-1}(x) = v_{n-1}(x)$\;}
$\tilde{v}_{n-1}(y_n) = 1$\;

\For{$x\in K_n \subseteq S_n\setminus\mathcal{B}(\mathcal{X}_{\rm goal},M h_n + d_n)$}{$(v_n(x),u_n(x)) \leftarrow {\rm \tt{VI}}(S_n,\tilde{v}_{n-1})$\;}

\For{$x\in S_n\setminus\big(K_n\cup\mathcal{B}(\mathcal{X}_{\rm goal},M h_n + d_n)\big)$}{$\displaystyle{v_n(x) = \min_{y\in\mathcal{B}(x,\alpha_{n-1})\cap S_{n-1}}v_{n-1}(y)}$\;}

\For{$x\in S_n\cap\mathcal{B}(\mathcal{X}_{\rm goal},M h_n + d_n)$}{$v_n(x) = \tilde{v}_{n-1}(x)$\;}
}
\caption{The $\iGame$ Algorithm}
\label{algorithm:iSaddle}
\end{algorithm}

\IncMargin{0.05in}

\begin{algorithm}[h] \small

$U_n \leftarrow U_{n-1} \cup {\tt Sample}(U,1)$\;

$v_n(x) \leftarrow 1 - e^{-\kappa_n} + e^{-\kappa_n}\displaystyle{\max_{w\in W}\min_{u\in U_n}\min_{y\in\mathcal{B}(x+h_nf(x,u,w),\alpha_n)\cap S_n}\tilde{v}_{n-1}(y)}$\;

$u_n(x) \leftarrow $ the solution to $u$ in the above step\;
\caption{${\rm \tt{VI}}(S_n,\tilde{v}_{n-1})$}
\label{algorithm:VI}
\end{algorithm}

\subsection{Performance analysis}

In this section, we analyze the asymptotic convergence properties of the iGame algorithm.

For the discrete fixed point $v^*_{n-1}$, we define the interpolated fixed point $\tilde{v}^*_{n-1} : S_n \rightarrow
[0,1]$ as follows: \begin{align*}\tilde{v}^*_{n-1}(x) = v^*_{n-1}(x),\quad x\in S_{n-1},\quad \tilde{v}^*_{n-1}(y_n) =
1.\end{align*} The following lemma shows the interpolation is consistent.

\begin{lemma} The following holds with probability one: \begin{align*}\|\min_{y\in\mathcal{B}(x,\alpha_n)\cap S_n}\tilde{v}_{n-1}^*(y) - \min_{z\in\mathcal{B}(x,\alpha_n)\cap S_n}v_n^*(z)\|_{\mathcal{X}}\rightarrow0.\end{align*}\label{lem3}\end{lemma}

It follows from Lemma~\ref{lem3} and the boundedness of $D$ that \begin{align*}&\sum_{\tau=1}^{D+1}\|\min_{y\in\mathcal{B}(x,\alpha_{n-\tau})\cap S_{n-\tau}}\tilde{v}_{n-\tau}^*(y)\\
&- \min_{z\in\mathcal{B}(x,\alpha_{n-\tau})\cap S_{n-\tau}}v_{n-\tau}^*(z)\|_{\mathcal{X}} \rightarrow 0.\end{align*} So there always exist a non-negative and non-increasing sequence $\{\gamma_n\}$ and a constant $C > 0$ such that \begin{align}&\sum_{\tau=0}^D\|\min_{y\in\mathcal{B}(x,\alpha_{n-\tau})\cap S_{n-\tau}}\tilde{v}_{n-\tau}^*(y)\nnum\\
&- \min_{z\in\mathcal{B}(x,\alpha_{n-\tau})\cap S_{n-\tau}}v_{n-\tau}^*(z)\|_{\mathcal{X}} \leq \gamma_n,\nnum\\
&\gamma_n\leq C \gamma_{n+1}.\label{e17}\end{align}

This following assumption requires that each state sample invokes the procedure $\tt VI$ at least once every $D+1$ iterations.

\begin{assumption} There is an integer $D\geq0$ such that for any $n\geq 1$, it holds that
$S_n\subseteq\cup_{s=0}^DK_{n+s}$.\label{asm2}
\end{assumption}

The following theorem is the core of this section and summarizes the convergence properties of the $\iGame$ algorithm.

\begin{theorem} Suppose that Assumptions~\ref{asm1} and~\ref{asm2} hold. Then the following properties hold for the $\iGame$ algorithm:
\begin{enumerate}
\item[(P1)] For any (small) $\epsilon>0$, it holds that:
\begin{align*}&\lim_{n\rightarrow+\infty}
\frac{1}{\gamma_n^{1-\epsilon}}\|v_n(x)-v_n^*(x)\|_{S_n}=0,\\
&\lim_{n\rightarrow+\infty}\frac{1}{\gamma_n^{1-\epsilon}}\nnum\\
&\big(\|\min_{y\in \mathcal{B}(x,\alpha_n)\cap S_n}v_n(y) - \min_{z\in \mathcal{B}(x,\alpha_n)\cap S_n}v_n^*(z)\|_{\mathcal{X}}\big) = 0.\end{align*}
\item[(P2)] The sequence $v_n$ converges to $v^*$ pointwise; i.e., for any $x\in \mathcal{X}_{\rm free}$, it holds that \begin{align*}v^*(x) = \lim_{n\rightarrow+\infty}\min_{y\in \mathcal{B}(x,d_n)\cap S_n}v_n(y).\end{align*} In addition, it holds that $\|v_n - v^*\|_{S_n} \rightarrow 0$.
\end{enumerate}\label{the1}
\end{theorem}

%\begin{remark} For any iteration~$n$, the interpolated value function $\hat{v}_n : \mathcal{X} \rightarrow [0,1]$ can be constructed as follows: $\hat{v}_n(x) = v_n(x)$ for $x\in S_n$; otherwise, $\hat{v}_n(x) = \min_{y\in \mathcal{B}(x,d_n)\cap S_n}v_n(y)$.\oprocend\label{rem2}\end{remark}

\subsection{Discussion}\label{sec:iSaddle-discussion}

We would like to make a couple remarks on Theorem~\ref{the1}. During the first paragraph of this section, we use $D = 0$ corresponding to synchronous updates. In (P1) of Theorem~\ref{the1}, $\|\min_{y\in
\mathcal{B}(x,\alpha_n)\cap S_n}v_n(y) - \min_{z\in \mathcal{B}(x,\alpha_n)\cap S_n}v_n^*(z)\|_{\mathcal{X}}$ is the distance to
$v^*$ of $\iGame$, and the quantity $\gamma_n$ represents an over-approximation of the distance to $v^*$ of
the multi-grid algorithm in~\cite{PC-MQ-PSP:99}. The relation of $\gamma_n\leq C \gamma_{n+1}$
means that $\gamma_{n+1}$ cannot be arbitrarily small in comparison with $\gamma_n$. One example is that $\{\gamma_n\}$
decreases in an exponential rate $a > 0$; i.e., $\gamma_n = e^{-a n}\gamma_0$. For this case, $\gamma_{n+1} =
e^{-a}\gamma_n$. Another example is $\gamma_n = \frac{\log n}{n}$, and then $\frac{\gamma_n}{\gamma_{n+1}} = \frac{\log
n}{\log (n+1)}\frac{n+1}{n}\leq \frac{n+1}{n}\leq2$. For the cases with $\gamma_n\leq C \gamma_{n+1}$, (P1) of Theorem~\ref{the1} demonstrates that, in the asymptotical sense, $\iGame$ is as optimal as the multi-grid method in~\cite{PC-MQ-PSP:99}. Here, we would like to emphasize that $\iGame$ only executes once the $\tt VI$ procedure on each grid. In contrast, the algorithm in~\cite{PC-MQ-PSP:99} has to completely determine the fixed point on each grid before the grid refinement, and thus, it may be difficult to determine when to refine the grids. Thus, the aggregate computational complexity of \emph{synchronous} $\iGame$; i.e., $D=0$, is lower than that of that in~\cite{PC-MQ-PSP:99}. We will numerically demonstrate this fact in Section~\ref{sec:simulation-HC}.

The $\iGame$ algorithm is built on the multi-grid method on Page 232 in~\cite{PC-MQ-PSP:99}. Here we would like to point out several distinctions. Firstly, as~\cite{SK-EF:11,SML-JJK:00,SML-JJK:01}, $\iGame$ utilizes random sampling which scales well in high-dimension spaces in comparison with \emph{a priori} discretization in~\cite{PC-MQ-PSP:99}. Secondly, $\iGame$ only performs the discrete operator $F_n$ \emph{only once} on each grid. This helps reduce the aggregate computational complexity of $\iGame$ over that in~\cite{PC-MQ-PSP:99}. This issue has been elaborated. Thirdly, the procedure $\tt VI$ can be performed on a subset of $S_n$ in contrast to the whole set $S_n$ in~\cite{PC-MQ-PSP:99}. This asynchronism feature itself is new and, more importantly, will allow for a significant improvement in the convergence rate in the next section. Fourthly, $\iGame$ is shown to be globally converge; that is, that algorithm converges from any initial values of $v_0(x)\in[0,1]$ for $x\in S_0$. This is in contrast to the local convergence in~\cite{PC-MQ-PSP:99} where $v_0(x) = 0$ for $x\in S_0$.

%A key feature of $\iGame$ is the \emph{anytime} property: a feasible and suboptimal solution is found quickly and the optimality of the solutions monotonically improves over time. This is ensured by two facts. Firstly, as~\cite{PC-MQ-PSP:99}, the grids are incrementally refined. Secondly, at each time instant~$n$, the algorithm \emph{slightly} refine the discretized set-valued dynamics by getting a \emph{single} new sample. It enables the algorithm to effectively exploit the past updates and \emph{monotonically} improve the optimality of the estimates. We will demonstrate the anytime property through a number of simulations. While there is no systematic way for the algorithm in~\cite{PC-MQ-PSP:99} to refine the grids.

\section{The $\iGameStar$ Algorithm: Cascade updates}\label{sec:iSaddle-star}

The new feature asynchronism of $\iGame$ offers a greater freedom for the online section of sample updates and opens up a number of opportunities to improve the computational efficiency. In this section, we will develop the variation of $\iGameStar$ which explicitly exploits the asynchronism.

To motivate $\iGameStar$, we first examine the behavior of $\iGame$. For each iteration, we define the directed graph $\GG_n \triangleq \{S_n, \EE_n\}$ such that $(x,y)\in \EE_n$ if $y\in \mathcal{B}(x, \alpha_n)\cap S_n$. That is, $v_n^*(x)$ is determined by the values of all $y$ such that $(x,y)\in\EE_n$.

In $\iGame$, the sample numbers are small for the initial iterations. So $v_{n}$ could approach $v_n^*(x)$ quickly. Let us assume this is the case at iteration~$n$; i.e., $v_n(x) \approx v_n^*(x)$ for $x\in S_n$. Due to the backward propagation of value iterations on $\GG_n$, the addition of a single point $y_n$ only affects the values of $v_{n+1}^*(x)$ for a subset of $S_n$. When $n$ is large, the variation $\|v_{n+1}^*(x)-v_n^*(x)\|$ could be small. If all the samples in $S_{n+1}$ are updated at iteration $n+1$, the decrease of the approximation error could be small. In contrast, a large amount of computation is required. The computational efficiency is particularly low when $n$ is large. This reveals that synchronous updates in $\iGame$ could waste a large amount of computation.

The above observation motivates the efficient utilization of asynchronism and thus $\iGameStar$ formally stated in Algorithm~\ref{algorithm:iSaddle-star}. In particular, $\iGameStar$ maintains a set of directed trees where the child $y$ of the sample $x$ is the one realizing the value of $x$ in $\tt{VI}^*$ at the last iteration. Note that the relation of parent-child is obtained as a byproduct of the procedure $\tt{VI}^*$. $\iGameStar$ employs the \emph{cascade update scheme} where the updates of the state samples are triggered if their children changes their estimates or they have not updated for $D$ consecutive iterations. That is, each new sample triggers a new round of value updates and the updates are performed backward along the paths of $\GG^*_n$. Instead of $\GG_n$, we use the subgraph $\GG^*_n\triangleq (S_n, \EE^*_n)$ of $\GG_n$ to define the parent-children relation since the generation of $\GG^*_n$ does not demand any additional computation.

\begin{figure}[h]
  \centering
  \includegraphics[width = .5\linewidth]{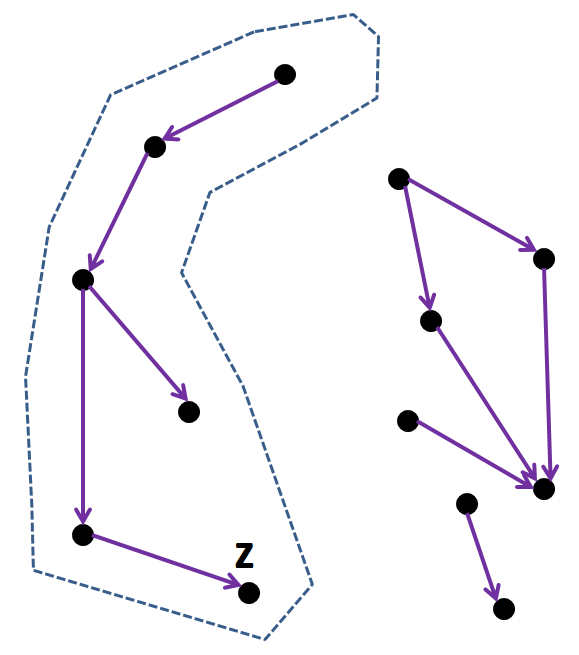}
  \caption{An illustrative example of the cascade update rule. The estimate change of state sample $z$ will trigger estimate updates backward on the tree containing $z$.} \label{fig2}
\end{figure}

\IncMargin{0.05in}

\begin{algorithm}[h] \small

\For{$x\in S_0 \subset \mathcal{X}$}{$v_0(x) = {\tt Sample}([0,1])$\;
${\tt Flag}(x) = 0$\;}
$n \leftarrow 1$\;

\While{$n < K$}{
$y_n \leftarrow {\tt Sample}(\mathcal{X},1)$\;
$S_n \leftarrow S_{n-1}\cup\{y_n\}$\;
$h_n \leftarrow d_n^{\frac{1}{1+\alpha}}$\;
$\alpha_n \leftarrow 2d_n + \ell h_n d_n + M\ell h_n^2$\;

%$X_{\rm near} \leftarrow {\tt Near}(S_{n-1},y_n,\alpha_n)$\;

\For{$x\in S_{n-1}$}{$\tilde{v}_{n-1}(x) \leftarrow v_{n-1}(x)$\;}
$\tilde{v}_{n-1}(y_n) \leftarrow 1$\;

$K_n \leftarrow \{y_n\}$\;

\For{$x\in S_{n-1}$}{\If{${\tt Flag}(x) == D$}{$K_n \leftarrow K_n \cup \{x\}$\;
Continue\;}
\For{${\tt Flag}({\tt Child}(x)) == 0$}{$K_n \leftarrow K_n \cup \{x\}$\;}}

\For{$x\in K_n\setminus\mathcal{B}(\mathcal{X}_{\rm goal},M h_n + d_n)$}{$(v_n(x),u_n(x),{\tt Child}(x)) \leftarrow {\tt VI}^*(S_n,\tilde{v}_{n-1})$\;
${\tt Flag}(x) \leftarrow 0$\;}

\For{$x\in S_n\setminus\big(K_n\cup\mathcal{B}(\mathcal{X}_{\rm goal},M h_n + d_n)\big)$}{$\displaystyle{z \leftarrow {\rm argmin}_{y\in\mathcal{B}(x,\alpha_{n-1})\cap S_{n-1}}v_{n-1}(y)}$\;
$v_n(x) \leftarrow v_{n-1}(z)$\;
${\tt Child}(x)\leftarrow z$\;
${\tt Flag}(x)\leftarrow{\tt Flag}(x) + 1$\;}

\For{$x\in S_n\cap\mathcal{B}(\mathcal{X}_{\rm goal},M h_n + d_n)$}{$v_n(x) \leftarrow \tilde{v}_{n-1}(x)$\;
${\tt Flag}(x) \leftarrow {\tt Flag}(x) + 1$\;}
}
\caption{The $\iGameStar$ Algorithm}
\label{algorithm:iSaddle-star}
\end{algorithm}

\IncMargin{0.05in}

\begin{algorithm}[h] \small

$U_n \leftarrow U_{n-1} \cup {\tt Sample}(U,1)$\;

$v_n(x) \leftarrow 1 - e^{-\kappa_n} + e^{-\kappa_n}\displaystyle{\max_{w\in W}\min_{u\in U_n}\min_{y\in\mathcal{B}(x+h_nf(x,u,w),\alpha_n)\cap S_n}\tilde{v}_{n-1}(y)}$\;

$(w_n(x),u_n(x)) \leftarrow$ the solution to $(w,u)$ in the above step\;

${\tt Child}(x) \leftarrow {\rm argmin}_{y\in\mathcal{B}(x+h_nf(x,u_n(x),w_n(x)),\alpha_n)\cap S_n}\tilde{v}_{n-1}(y)$;

\caption{${\rm \tt{VI}^*}(S_n,\tilde{v}_{n-1})$}
\label{algorithm:VI}
\end{algorithm}

$\iGameStar$ shares the same asymptotic convergence as $\iGame$, and the property is summarized as follows:

\begin{theorem} Suppose that Assumption~\ref{asm1} holds. Then the properties of (P1) and (P2) in Theorem~\ref{the1} hold for the $\iGameStar$ Algorithm.\label{the4}
\end{theorem}

\begin{proof} Note that Assumption~\ref{asm2} in Theorem~\ref{the1} is enforced. Thus the proof is identical to that of Theorem~\ref{the1}.
\end{proof}

\begin{remark} The cascade update rule is not unique. The one we choose in $\iGameStar$ is completely built on the computation results of the previous iteration and does not incur any additional computation effort. Intuitively, $\iGameStar$ is at least as fast as $\iGame$. It turns out that $\iGameStar$ is significantly faster than $\iGame$, shown in Section~\ref{sec:simulation}.\oprocend\label{rem3}
\end{remark}

\section{Simulation Results}\label{sec:simulation}

This section describes simulation experiments conducted to assess the performance of $\iGame$ and $\iGameStar$.

\subsection{Fence Escape Differential Game}\label{sec:simulatuin-fence}

The first set of results focus on a simple zero-sum differential game with a known analytic solution and a discontinuous value function, referred to here as fence escape. In this game, two agents, a pursuer and an evader, move along opposite sides of a straight fence segment extending from $x=0$ to $x=10$. The positions of the pursuer and evader are respectively denoted by $x_p$ and $x_e$. The agents can directly command their velocities, $u_p = \dot{x}_p ~~ u_e = \dot{x}_e$, which are both bounded in magnitude by 1. The game terminates when the evader passes either end of the fence, but
the evader is blocked by the pursuer and may not exit the fence region if $|x_e-x_p| \leq 1$. Formally, the goal set for the evader is:

\begin{equation*}
 \mathcal{X}_{\rm goal} \, : \; \; \; ((x_e < 0) \lor (x_e > 10)) \wedge (|x_e-x_p| > 1)
\end{equation*}

The objective of the evader is to reach this set in minimum time. The optimal value function for this game has a simple analytic solution shown in Figure~\ref{fig:fence_analytic_value}.

The evolution of the $\iGame$ value function approximation for the fence escape game is shown in Figure~\ref{fig:fence_IRVI_value}. Here we see the initial crude approximation to the value function become refined over time, as well as convergence to the discontinuous value function.

\begin{figure}[h]
  \centering
  \subfloat[100 points,
  111ms]{\label{fig:fence_IRVI_value100}\includegraphics[width=0.5\columnwidth]{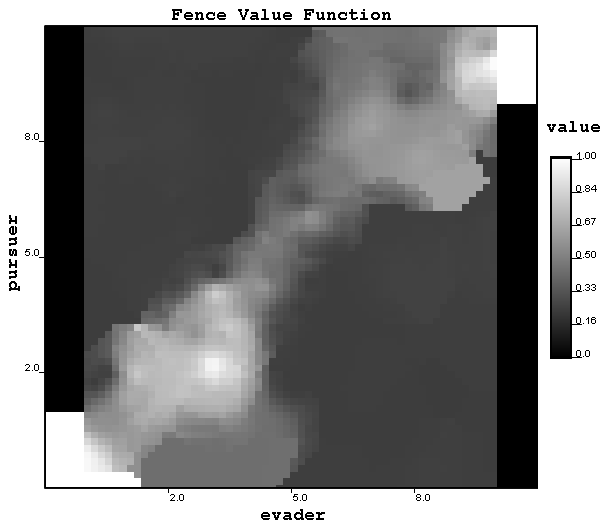}}
  ~
  \subfloat[500 points,
  848ms]{\label{fig:fence_IRVI_value500}\includegraphics[width=0.5\columnwidth]{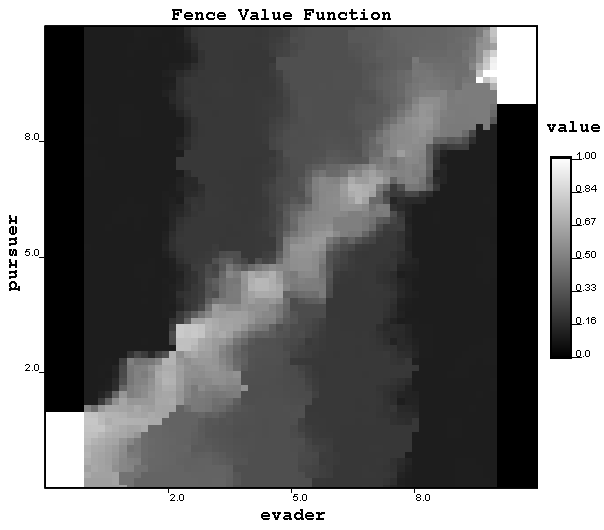}}
  \\
  \subfloat[1000 points,
  2.0s]{\label{fig:fence_IRVI_value1000}\includegraphics[width=0.5\columnwidth]{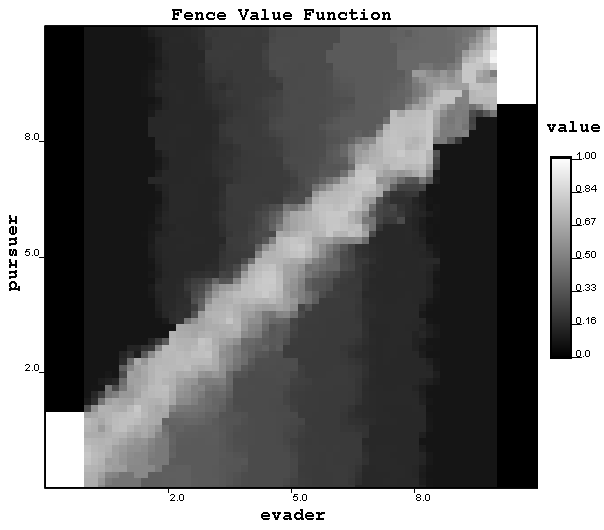}}
  ~
  \subfloat[2000 points,
  6.0s]{\label{fig:fene_IRVI_value2000}\includegraphics[width=0.5\columnwidth]{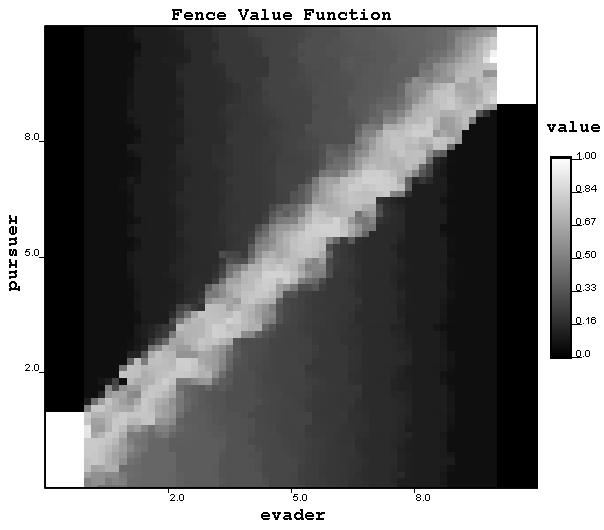}}
  \\
  \subfloat[6000 points,
  50.6s]{\label{fig:fence_IRVI_value6000}\includegraphics[width=0.5\columnwidth]{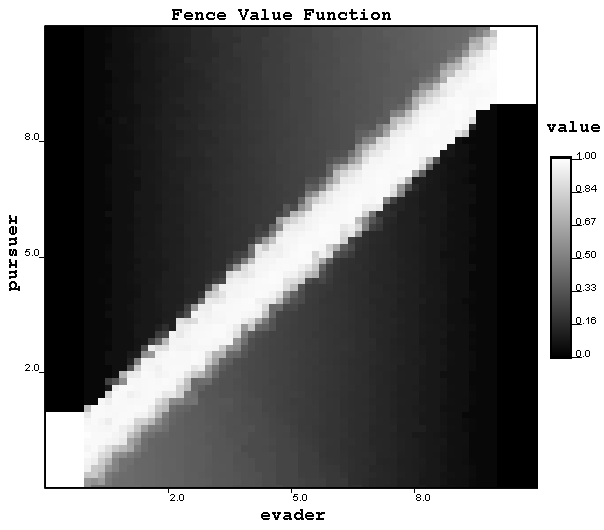}}
  ~
  \subfloat[analytic
  solution]{\label{fig:fence_analytic_value}\includegraphics[width=0.5\columnwidth]{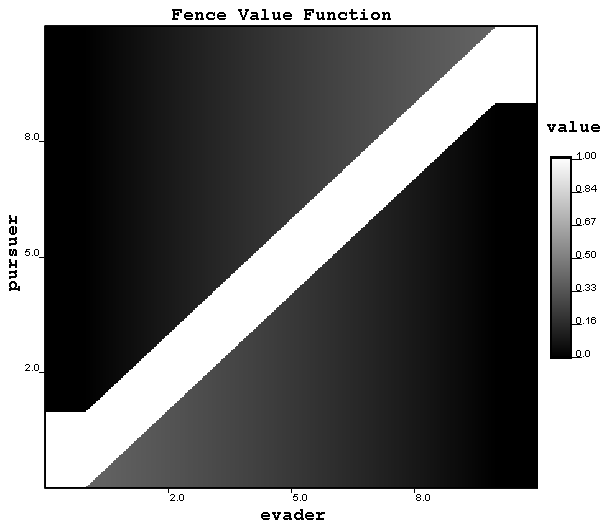}}
  \\

  \caption{Progression of the value function approximation computed by $\iGame$, for the fence escape game, indicating the number of samples, and the elapsed computation time. Progressive refinement of the value function and convergence to the analytic solution, shown in figure \ref{fig:fence_analytic_value}, is seen.}
  \label{fig:fence_IRVI_value}
\end{figure}

\subsection{Homicidal Chauffeur}\label{sec:simulation-HC}

Now we consider a variation of the classic homicidal chauffeur problem~\cite{TB-GO:99}, in which a fast but less agile pursuer seeks to bring a slower, more manoeuvrable evader within some capture distance in minimum time, while the evader tries to reach some greater escape distance. For the pursuer we use a simplistic model of vehicular dynamics, in which heading rate is commanded directly, while the evader is kinematically unconstrained except for a maximum velocity. This results in a game in $\real^5$ with dynamics:
\begin{align*}
& \dot{x}_p = v_p \cos{\theta}, ~~~~~~~ \dot{y}_p = v_p \sin{\theta}, ~~~~~~~ \dot{\theta} = u_p \\
& \dot{x}_e = v_e \cos{u_e}, ~~~~~~~ \dot{y}_e = v_e \sin{u_e},
\end{align*}
Where $(x_p, y_p, \theta)$ defines the state of the pursuer, who commands angular rate, $u_p ~,~
|u_p| < \omega$, and $(x_e,y_e)$ defines the state of the evader, who commands angle, $u_e$, directly. By expressing the dynamics in a coordinate system fixed to the pursuer, this problem may be reduced to 2 dimensions, in which the dynamics are:
\begin{align*}
& \dot{x} = u_p y + v_e \cos{u_e} - v_p \\
& \dot{y} = -u_p x - v_e \sin{u_e} \\
\end{align*}
Where $q = (x,y)$ represents the location of the evader in a coordinate system fixed to the pursuer - translated by $(x_p,y_p)$ and rotated by angle $\theta$. In this formulation of the game, the pursuer wins if $||q||_{\infty} < r_p$, which represents intersection of the evader with a square pursuer, and the evader wins if $|q| > r$. The parameters used in the simulation tests are:

\begin{table}[h]
\centering
\begin{tabular}{|l|l||l|l||l|l||l|l||l|l|}
\hline
$\omega$ & 5 & $v_e$ & 0.5 & $v_p$ & 1 & $r$ & 1 & $r_p$ & 0.05
\tabularnewline
\hline
\end{tabular}
\end{table}

Firstly, we examine the performance of synchronous $\iGame$ ($D = 0$) and $\iGameStar$ on the 2-dimensional problem. Figure~\ref{fig:HC-comparison-error} shows the temporal evolutions of the approximation errors generated by the multi-grid, $\iGame$ and $\iGameStar$ Algorithms. Each curve corresponds to the average of a number of trials for a specific algorithm. The trial number of $\iGameStar$ is $100$. Due to their relatively slow convergence, we only simulate the multi-grid method and $\iGame$ Algorithm for $5$ and $10$ times, respectively. The standard deviations of each algorithm is also shown in Figure~\ref{fig:HC-comparison-error}. In order to compute the approximation errors, we use the solution computed by a $200\times200$ grid as the benchmark. It is noted that logarithmic scaling is used in the figure.

Figure~\ref{fig:HC-comparison-error} clearly shows that the convergence of $\iGame$ is faster than that of the multi-grid method, and $\iGameStar$ is significantly faster than $\iGame$. In Figure~\ref{fig:HC-chart}, we compare the average computation times three algorithms take to reach the approximation errors of $0.1$, $0.08$, $0.06$ and $0.04$. Figure~\ref{fig:HC-chart} confirms the superior convergence rates of $\iGame$ and $\iGameStar$ over that of the multi-grid method. The evolution of the $\iGame$ value function approximation for the fence escape game is shown in Figure~\ref{fig:HC-iSaddle-star-incremental}.

\begin{figure}[h]
	\vspace{0cm}
	   \center{\includegraphics[width=\columnwidth]
	       {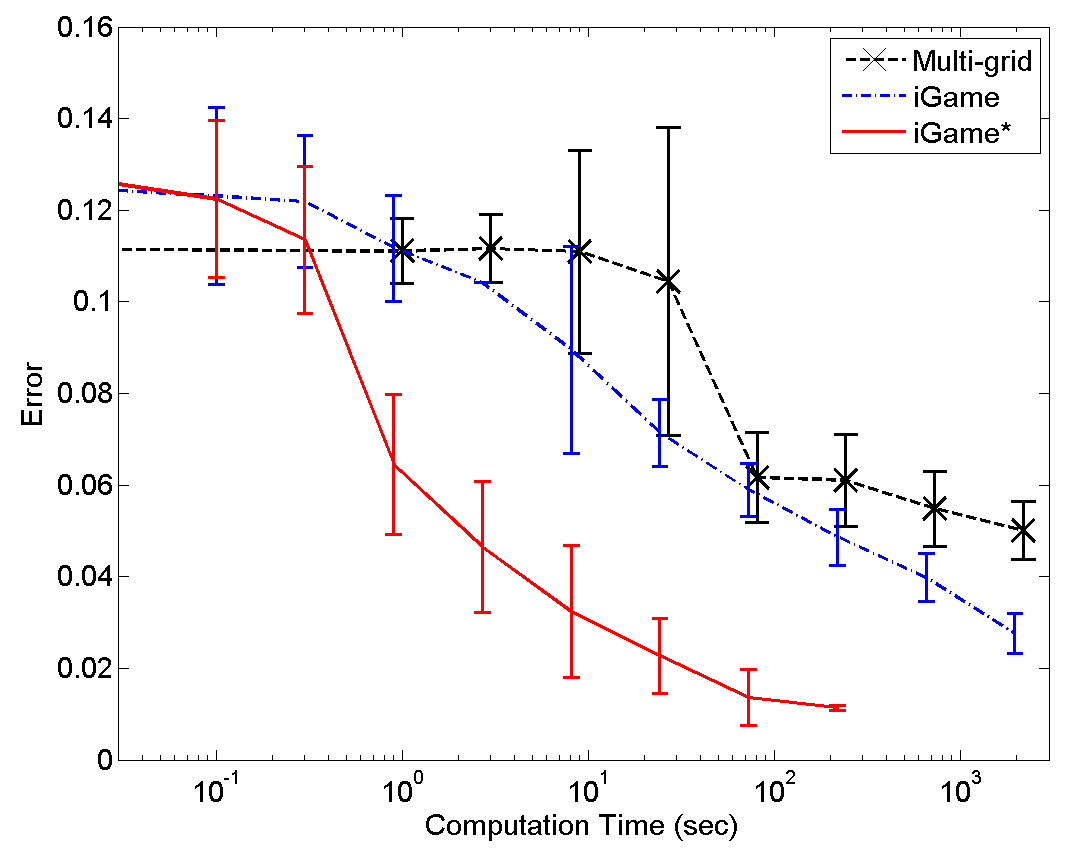}}
	  \vspace{0.0in}
	  \caption{The average errors and standard deviations of the multi-grid, $\iGame$ and $\iGameStar$ Algorithms}
	  \label{fig:HC-comparison-error}
\end{figure}

\begin{figure}[h]
  \centering
  \subfloat[]{\label{fig:HC_igame_value500}\includegraphics[width=0.5\columnwidth]{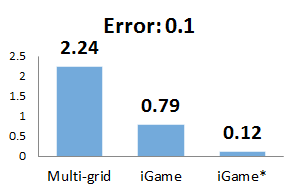}}
  ~
  \subfloat[]{\label{fig:HC_igame_value1000}\includegraphics[width=0.5\columnwidth]{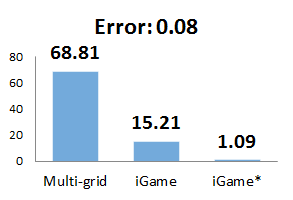}}
  \\
  \subfloat[]{\label{fig:HC_igame_value2000}\includegraphics[width=0.5\columnwidth]{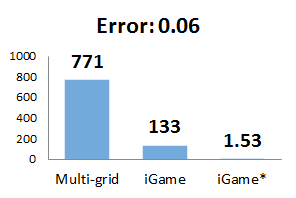}}
  ~
  \subfloat[]{\label{fig:HC_igame_value6500}\includegraphics[width=0.5\columnwidth]{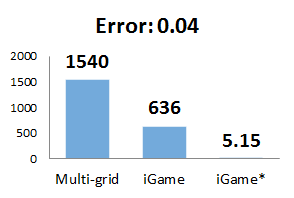}}
  \caption{The comparison of average elapsing times (seconds) to reach different approximation errors: multi-grid (left), $\iGame$ (middle), $\iGameStar$ (right)}
  \label{fig:HC-chart}
\end{figure}

\begin{figure}[h]
	\vspace{0cm}
	   \center{\includegraphics[width=\columnwidth]
	       {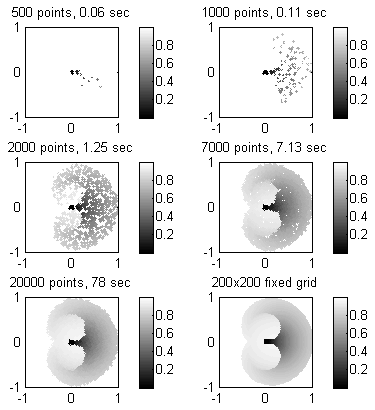}}
	  \vspace{0.0in}
	  \caption{Progression of the value function approximation computed by $\iGameStar$, for the homicidal chauffeur game, indicating the number of points in the set, and the elapsed computation time.}
	  \label{fig:HC-iSaddle-star-incremental}
\end{figure}

\begin{figure}[h]
	   \center{\includegraphics[width=\columnwidth]
	       {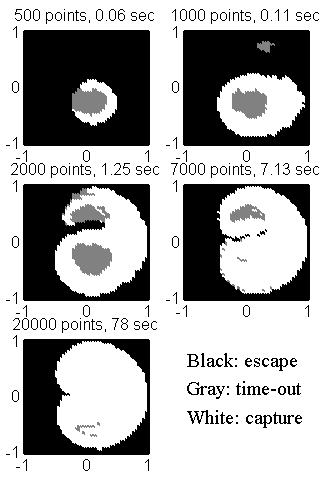}}
\caption{Progression of capture and escape regions for the homicidal chauffeur problem, with the evader using a 50x50 fixed grid approximation, and the pursuer using $\iGameStar$.} \label{fig:HC-realgame}
\end{figure}

Secondly, we examine the performance of $\iGameStar$ when the pursuer uses $\iGameStar$ and the evader uses a fixed-grid approximation based on a $50\times50$ grid iterated to convergence. Figure~\ref{fig:HC-realgame} shows
the outcome of the game from initial states on a $50\times50$ grid. For practical purposes, a maximum time approximately $10$ times larger than a reasonable capture time was used for the duration of the simulated games, with the outcome of games reaching this threshold referred to as ``timeout''.
Qualitatively, games so terminating generally exhibit oscillatory behaviour, with the pursuer attempting to capture the evader, who narrowly escapes, the pursuer turning around to try again, the evader escaping again, and so on. Thus, while some games labeled as timeout may in fact have terminated at some later time, it is thought that most will behave this way and so timeout can be thought of also in this sense as a stalemate.

These tests demonstrate the performance of $\iGameStar$ against a fixed grid opponent improving as the number of points increases, to the point where the evader is captured from nearly all states where capture is possible, and some states from which escape is possible. The latter case is not entirely obvious in the figures, mostly due to the fact that the 50x50 fixed grid still yields a fairly good approximation. Importantly, these results also show the benefit of using $\iGameStar$ in an online setting - early on, when few points are used, games starting from many states terminate in timeout, however, this outcome is not truly a loss for the pursuer, but allows more time to incrementally improve the policy. As the policy is improved, the pursuer is better able to capture the evader, and states which initially would have resulted in timeout result in capture. In other words, so long as the pursuer is able to prevent the evader from escaping, the policy can be improved, and capture is inevitable.

\subsection{Implementation Notes}

%All simulation tests were written in Java version 1.6.0 and run on a machine running Windows 7, with 6GB RAM and an Intel(R) Core(TM) i3 processor running at 2.93 GHz.

All simulation tests were written in Java version 1.7.0 and run on a machine running Windows 7, with 8GB RAM and an Intel(R) Core(TM) i7 processor running at 2.80 GHz.

For all simulation tests described in this document, the optimal feedback control policies were determined based on the value function approximation, for all the algorithms, by solving a single-step min-max problem over a discrete set of inputs, as in the value update.

\section{Conclusions}

We have investigated two anytime computation algorithms, $\iGame$ and $\iGameStar$, for the approach-evasion differential game. Their asymptotic convergence properties have been formally analyzed. Through a number of numerical simulations, we have demonstrated their anytime property and superior computational efficiency over the state of the art. Future directions include stochastic differential games and multi-player differential games.

%\bibliographystyle{plain}
%\bibliography{MZ}

\section{Appendix}

In this section, we will provide the complete proof for Lemma~\ref{lem3} and Theorem~\ref{the1}.

\subsection{Preliminary results}

Let us start with a set of generic preliminary results. We consider finite grids $X_d\subset \mathcal{X}$, $U_d\subset U$ and $W_d\subset W$ such that \begin{align}&\forall x\in \mathcal{X},\;\; \exists x'\in X_d,\;\; {\rm s.t.}\;\; \|x - x'\| \leq d,\nnum\\
&\forall u\in U,\;\; \exists u'\in U_d,\;\; {\rm s.t.}\;\; \|u - u'\| \leq d,\nnum\\
&\forall w\in W,\;\; \exists w'\in W_d,\;\; {\rm s.t.}\;\; \|w - w'\| \leq d.\label{e41}\end{align} Note that $X_d$, $U_d$ and $W_d$ are not necessarily regular lattices. The following lemma show the state discretization is non-expansive.

\begin{lemma}[Non-expansiveness of state discretization] Consider any two functions $v,\bar{v} : X_d \rightarrow [0,1]$.
\begin{enumerate}
\item[(T1)] The following holds for any finite set $X_{d'}\subseteq X_d$ and any scalar $d''\geq d'$: \begin{align}&\|\min_{y\in \mathcal{B}(x,d'')\cap X_{d'}}v(y)-\min_{z\in \mathcal{B}(x,d'')\cap X_{d'}}\bar{v}(z)\|_{\mathcal{X}}\nnum\\
    &\leq\|v-\bar{v}\|_{X_d},\label{e43}\\
    &\|\max_{y\in \mathcal{B}(x,d'')\cap X_{d'}}v(y)-\max_{z\in \mathcal{B}(x,d'')\cap X_{d'}}\bar{v}(z)\|_{\mathcal{X}}\nnum\\
    &\leq\|v-\bar{v}\|_{X_d}.\label{e46}\end{align}
\item[(T2)] The following holds for any $\alpha \geq d \geq d'$:
 \begin{align}&\|\min_{y\in \mathcal{B}(x,\alpha)\cap X_d}v(y)-\min_{z\in \mathcal{B}(x,\alpha)\cap X_{d'}}\bar{v}(z)\|_{\mathcal{X}}\nnum\\
    &\leq \|\min_{y\in \mathcal{B}(x,d)\cap X_d}v(y)-\min_{z\in \mathcal{B}(x,d')\cap X_{d'}}\bar{v}(z)\|_{\mathcal{X}}.\label{e48}\end{align}
\end{enumerate}\label{lem1}\end{lemma}

\begin{proof} (T1) Pick any $x\in \mathcal{X}$, and let $\min_{y\in \mathcal{B}(x,d'')\cap X_{d'}}v(y) = v(y')$ and $\min_{z\in \mathcal{B}(x,d'')\cap X_{d'}}\bar{v}(z) = \bar{v}(z')$ for some $y',z'\in X_{d'}$. Since $y',z'\in \mathcal{B}(x,d'')\cap X_d$, we then have \begin{align}v(y') \leq v(z'),\quad \bar{v}(z') \leq \bar{v}(y').\label{e25}\end{align}

We now move to show $\|v(y') - \bar{v}(z')\| \leq \max\{\|v(z') - \bar{v}(z')\|,\|\bar{v}(y') - v(y')\|\}$. Towards this end, we have \begin{align*}&v(y') - \bar{v}(z') \leq v(z') - \bar{v}(z')\nnum\\
&\leq \max\{\|v(z') - \bar{v}(z')\|,\|\bar{v}(y') - v(y')\|\},\end{align*} where the first inequality is a direct result of~\eqref{e25}. Analogously, we have \begin{align*}&\bar{v}(z') - v(y') \leq \bar{v}(y') - v(y')\nnum\\
&\leq \max\{\|v(z') - \bar{v}(z')\|,\|\bar{v}(y') - v(y')\|\}.\end{align*}

Combine the above two cases, and then we have \begin{align}&\|\min_{y\in \mathcal{B}(x,d')\cap X_{d'}}v(y)-\min_{z\in \mathcal{B}(x,d')\cap X_{d'}}\bar{v}(z)\| = \|v(y') - \bar{v}(z')\|\nnum\\
&\leq \max\{\|v(z') - \bar{v}(z')\|,\|\bar{v}(y') - v(y')\|\}\leq\|v-\bar{v}\|_{X_d},\label{e16}\end{align} where the last inequality uses the fact that $y',z'\in X_{d'}\subseteq X_d$. Since~\eqref{e16} holds for any $x\in \mathcal{X}$, then we reach the desired result of~\eqref{e43}. One can show~\eqref{e46} through analogous arguments.

(T2) Note that $\alpha \geq d \geq d'$. Then there are $y',z'\in \mathbb{B}(x,\alpha)$ such that $\mathbb{B}(y',d),\mathbb{B}(z',d')\subseteq \mathbb{B}(x,\alpha)$ and \begin{align*}&\min_{y\in \mathcal{B}(x,\alpha)\cap X_d}v(y) = \min_{y\in \mathcal{B}(x',d)\cap X_d}v(y),\\&\min_{z\in \mathcal{B}(x,\alpha)\cap X_{d'}}\bar{v}(z) = \min_{z\in \mathcal{B}(z',d')\cap X_{d'}}\bar{v}(z).\end{align*} The remainder of the proof can be finished via following the same lines in (T1), and thus omitted. \end{proof}

For time discretization, we choose the size $h > 0$ where $h > d$. After time and state discretization, we then obtain a discrete-time game played on $X_d$. We now proceed to find the value function of this fully-discrete game. This will be achieved by applying the Dynamic Programming Principle. Towards this end, we define the discrete function $v : X_d \rightarrow [0,1]$, and the discrete operator $\FF_{\rho} : X_d \rightarrow [0,1]$ as follows: for $x\in X_d\setminus \mathcal{B}(\mathcal{X}_{\rm goal},Mh+d)$, \begin{align}\FF_{\rho}\circ v(x) &=
e^{-(h-d)}\displaystyle{\max_{w\in W_d}\min_{u\in U_d, y\in \mathcal{B}(x+hf(x,u,w),\alpha)\cap X_d}}v(y)\nnum\\&+ 1 - e^{-(h-d)},\nnum\end{align} and, for $x\in X_d\cap \mathcal{B}(\mathcal{X}_{\rm goal},Mh+d)$, $\FF_{\rho}\circ v(x) = v(x)$.
Since $\alpha > d$, $\mathcal{B}(x+hf(x,u,w),\alpha)\cap X_d \neq \emptyset$ for any $(u,w)\in U\times W$, and thus the operator $\FF_{\rho}$ is well-defined. Let $\FF^n_{\rho}$ be the composition of $n$ of $\FF_{\rho}$; i.e., $\FF^n_{\rho} \triangleq \FF_{\rho}\circ\cdots\circ\FF_{\rho}$.

\begin{remark} Define the discrete operator $\GG_{\rho} : X_d \rightarrow [0,1]$: for $x\in X_d\setminus \mathcal{B}(\mathcal{X}_{\rm goal},Mh+d)$, \begin{align}\mathcal{G}_{\rho}\circ T(x) &=
h - d\nnum\\
&+ \displaystyle{\max_{w\in W_d}\min_{u\in U_d} \min_{y\in \mathcal{B}(x+hf(x,u,w),\alpha)\cap X_d}}T(y),\nnum\end{align} and, for $x\in X_d\cap \mathcal{B}(\mathcal{X}_{\rm goal},Mh+d)$, $\mathcal{G}_{\rho}\circ v(x) = v(x)$. The operator $\GG_{\rho}$ serves as the core of the numerical schemes on Page 232 in~\cite{PC-MQ-PSP:99}. It is not difficult to verify that $\FF_{\rho} = \Psi\circ\GG_{\rho}\circ\Psi^{-1}$.\oprocend\label{rem1}
\end{remark}

The following theorem summarizes the properties of $\FF_{\rho}$ and $v^*$. It is well-known that the contraction property holds for discounted optimal control; e.g., in~\cite{DPB:95}.

\begin{theorem}[Properties of $\FF_{\rho}$ and $v^*$] Suppose that Assumption~\ref{asm1} holds. The value function $v^*$ is lower semi-continuous. And the discrete operator $\FF_{\rho}$ is a contraction mapping with the constant $e^{-(h-d)}\in(0,1)$ in the following way: consider $X_{d'}\subseteq X_d$ with $d'\leq d$ and any $v\;:\;X_d \rightarrow [0,1]$, the following holds:
\begin{align*}&\|\min_{z\in \mathcal{B}(x,d')\cap X_{d'}}\FF_{\rho}\circ v(z) - \min_{z\in \mathcal{B}(x,d')\cap X_{d'}}\FF_{\rho}\circ v^*(z)\|_{\mathcal{X}}\nnum\\
&\leq\|\FF_{\rho}\circ v-\FF_{\rho}\circ v^*\|_{X_d} \leq e^{-(h-d)}\nnum\\
&\times\|\min_{z\in \mathcal{B}(x,\alpha)\cap X_d}v(z)-\min_{y\in \mathcal{B}(x,\alpha)\cap X_d}v^*(y)\|_{\mathcal{X}}\nnum\\
&\leq e^{-(h-d)}\|v- v^*\|_{X_d}.\end{align*}\label{the2}
\end{theorem}

\begin{proof} Under Assumption~\ref{asm1}, The function $T^*$ is lower semi-continuous by Theorem 5.2 and Corollary 5.6 in~\cite{PC-MQ-PSP:99} with $\mathcal{X}_{\rm free} = K$ and $\mathcal{X}_{\rm target} = C$. Recall $v^* = \Psi(T^*)$, and $\Psi$ is continuous and bijective. This in conjunction with the lower semi-continuity of $T^*$ implies that of $v^*$.

Now we proceed to verify the contraction property of $\FF_{\rho}$. Choose any $x\in X_d\setminus \mathcal{B}(\mathcal{X}_{\rm goal},Mh+d)$. Let $(u(x),w(x))$ (resp. $(\bar{u}(x),\bar{w}(x))$) be a solution to attain the minimum and the maximum in $\FF_{\rho}\circ v(x)$ (resp. $\FF_{\rho}\circ v_{h,d}^*(x)$) respectively. Let $Y(x) \triangleq x + hf(x,u(x),w(x))$ and $\bar{Y}(x) \triangleq x + hf(x,\bar{u}(x),\bar{w}(x))$. Then we have \begin{align}&\min_{y\in\mathcal{B}(Y(x),\alpha)\cap X_d}v(y)\leq \min_{y\in\mathcal{B}(\bar{Y}(x),\alpha)\cap X_d}v(y),\nnum\\
&\min_{z\in\mathcal{B}(\bar{Y}(x),\alpha)\cap X_d}v^*(z)\leq \min_{z\in\mathcal{B}(Y(x),\alpha)\cap X_d}v^*(z).\label{e51}\end{align}

By following the same lines in (T1) of Lemma~\ref{lem1}, one can show that \begin{align}&\|\min_{y\in\mathcal{B}(Y(x),\alpha)\cap X_d}v(y) - \min_{z\in\mathcal{B}(\bar{Y}(x),\alpha)\cap X_d}v^*(z)\|\nnum\\
&\leq \max\{\|\min_{y\in\mathcal{B}(Y(x),\alpha)\cap X_d}v(y) - \min_{z\in\mathcal{B}(Y(x),\alpha)\cap X_d}v^*(z)\|,\nnum\\
&\|\min_{y\in\mathcal{B}(\bar{Y}(x),\alpha)\cap X_d}v(y) - \min_{z\in\mathcal{B}(\bar{Y}(x),\alpha)\cap X_d}v^*(z)\|\}.\label{e52}\end{align} The relation~\eqref{e52} implies the following: \begin{align}&\|\FF_{\rho}\circ v(x) - \FF_{\rho}\circ v^*(x)\|_{X_d} \leq e^{-(h-d)}\nnum\\
&\times\|\min_{z\in \mathcal{B}(x,\alpha)\cap X_d}v(z)-\min_{y\in \mathcal{B}(x,\alpha)\cap X_d}v^*(y)\|_{\mathcal{X}}.\label{e26}\end{align}

On the other hand, one can see the following: \begin{align}\|\FF_{\rho}\circ v-\FF_{\rho}\circ v_{h,d}^*\|_{X_d\cap\mathcal{B}(\mathcal{X}_{\rm goal},Mh+d)} =0.\label{e36}\end{align}

Combine~\eqref{e26} and~\eqref{e36}, and it renders the following: \begin{align}&\|\min_{z\in \mathcal{B}(x,d')\cap X_{d'}}\FF_{\rho}\circ v(z) - \min_{z\in \mathcal{B}(x,d')\cap X_{d'}}\FF_{\rho}\circ v^*(z)\|_{\mathcal{X}}\nnum\\
&\leq\|\FF_{\rho}\circ v-\FF_{\rho}\circ v_{h,d}^*\|_{X_d} \leq e^{-(h-d)}\nnum\\
&\times\|\min_{z\in \mathcal{B}(x,\alpha)\cap X_d}v(z)-\min_{y\in \mathcal{B}(x,\alpha)\cap X_d}v_{h,d}^*(y)\|_{\mathcal{X}}\nnum\\
&\leq e^{-(h-d)}\|v- v_{h,d}^*\|_{X_d}.\label{e37}\end{align} where the first and third inequalities are  direct results of~\eqref{e43} in Lemma~\ref{lem1}. We establish the result by noting the fixed point property of $\FF_{\rho}\circ v_{h,d}^*(x) = v_{h,d}^*(x)$.
\end{proof}

It is noticed that $\FF_{\rho}$ being a contraction mapping will play an important role in our subsequent analysis. Since $T^*(x)$ is potentially infinite, then $\GG_{\rho}$ in~\cite{PC-MQ-PSP:99} may not be a contraction mapping.

\subsection{Preliminary results for Lemma~\ref{lem3} and Theorem~\ref{the1}}

We turn our attention to the $\iGame$ Algorithm in this section. Before showing Theorem~\ref{the1}, we provide a set of instrumental results.

\subsubsection{Sample densities}

The following lemma characterizes the decreasing rate of $d_n$ under the uniform sampling through a pair of lower and upper bounds for $d_n$.

\begin{lemma} The following properties hold with probability one:
\begin{enumerate}
\item[(S1)] There is $N_s \geq 1$ such that $d_n \leq D_s(\frac{\log n}{n})^{\frac{1}{N}}$ for all $n\geq N_s$ with probability one.
\item[(S2)] There is a pair of $n_s \geq 1$ and $d_s > 0$ such that $d_n \geq d_s(\frac{1}{n})^{\frac{1}{N}}$ surely for all $n\geq n_s$.
\end{enumerate}\label{lem2}\end{lemma}

%We first show (S1) which is analogous to Lemma 52 in~\cite{SK-EF:11}.

\begin{proof} Define the event $E_n$ that $d_n > D_s(\frac{\log n}{n})^{\frac{1}{N}}$; i.e., there is $z_n\in \mathcal{X}_{\rm free}$ such that $\mathcal{B}(z_n,D_s(\frac{\log n}{n})^{\frac{1}{N}})\cap S_n = \emptyset$. Note that the Lebesgue measure of the ball $\mathcal{B}(z_n,D_s(\frac{\log n}{n})^{\frac{1}{N}})$ is $\mu(\mathcal{B}(z_n,D_s(\frac{\log n}{n})^{\frac{1}{N}})) = C_N \big(D_s(\frac{\log n}{n})^{\frac{1}{N}}\big)^N$. With this, we can estimate $\mathbb{P}(E_n)$ as follows: \begin{align}&\mathbb{P}(E_n)
= \big(1-\frac{\mu(\mathcal{B}(z_n,D_s(\frac{\log n}{n})^{\frac{1}{N}}))}{\mu(\mathcal{X}_{\rm free})}\big)^{n-1}\nnum\\
&= (1-\frac{C_N D_s^N}{\mu(\mathcal{X}_{\rm free})}\frac{\log n}{n})^{n-1}\leq e^{-\gamma\frac{(n-1)\log n}{n}}\leq n^{-\gamma+\frac{\gamma}{n}}.\nnum\end{align} Since $\gamma > 2$, then $\sum_{n=1}^{+\infty}\mathbb{P}(E_n) < +\infty$. By the Borel-Cantelli lemma in~\cite{SR:98}, we have $\mathbb{P}(\limsup_{n\rightarrow+\infty}E_n)=0$. It establishes (S1).

We now move to show (S2) by contradiction. Assume that for any pair of $n_s \geq 1$ and $d_s > 0$, there is $K \geq n_s$ such that $d_K < d_s(\frac{1}{K})^{\frac{1}{N}}$. By the definition of $d_K$, we know that $\cup_{\ell=1}^K\mathcal{B}(x_{\ell},d_K)\supseteq \mathcal{X}$. It is noticed that \begin{align*}&\mu(\cup_{\ell=1}^K\mathcal{B}(x_{\ell},d_K)\leq \sum_{\ell=1}^K\mu(\mathcal{B}(x_{\ell},d_s(\frac{1}{K})^{\frac{1}{N}}))\\
&= KC_N(d_s(\frac{1}{K})^{\frac{1}{N}})^N = C_N d_s^N.\end{align*} The above holds for any $d_s > 0$. Choose $d_s$ such that $C_N d_s^N < \mu(\mathcal{X})$. It contradicts that $\cup_{\ell=1}^K\mathcal{B}(x_{\ell},d_K)\supseteq \mathcal{X}$. We then reach the property of (S2).
\end{proof}

With the aid of Lemma~\ref{lem2}, the following lemma characterizes the non-summability of the sequence $\{\kappa_n\}$ by recalling $\kappa_n = h_n - d_n$.

\begin{lemma} Suppose Assumption~\ref{asm2} holds, then the sequence $\{\kappa_n\}$ is not summable.\label{lem8}
\end{lemma}

\begin{proof} Recall that Lemma~\ref{lem2} proves that $d_s(\frac{1}{n})^{\frac{1}{N}} \leq d_n \leq D_s(\frac{\log n}{n})^{\frac{1}{N}}$ for $n\geq \max\{n_s,N_s\}$. Since $\frac{d_n}{h_n}\rightarrow0$, there is $n_1\geq \max\{n_s,N_s\}$ such that the following holds for $n\geq n_1$: \begin{align}\kappa_n &= h_n - d_n = h_n(1-\frac{d_n}{h_n})\nnum\\
&\geq \frac{h_n}{2}\geq \frac{1}{2}(d_s)^{\frac{1}{1+\alpha}}
(\frac{1}{n})^{\frac{1}{N(1+\alpha)}}.\nnum\end{align} Let $\phi_n = \frac{1}{2}(d_s)^{\frac{1}{1+\alpha}}(\frac{1}{n})^{\frac{1}{N(1+\alpha)}}$, and then $\{\phi_n\}$ serves as a lower bound of $\{\kappa_n\}$ after a finite time; i.e., there is $N'\geq1$ such that $\kappa_n \geq \phi_n,\quad \forall n\geq N'$.

Let $\tau$ to be the smallest integer such that $\tau D \geq N'$. It is not difficult see that $\{\phi_n\}_{n\in \Psi,n\geq \tau D}$ and thus $\{\kappa_n\}_{n\in \Psi}$ are not summable. We then reach the desired result.
\end{proof}

The quantity $d_n$ in Lemma~\ref{lem2} serves as a measure of the sample density of $S_n$. The following lemma provides another characterization.

\begin{lemma} Consider any pair of $x,y\in \mathcal{X}$ with $\|x-y\|\leq\alpha_n$ and $x\neq y$. Then it holds that \begin{align}\mathcal{B}(x,\alpha_n)\cap\mathcal{B}(y,\alpha_n)\cap \big(S_n\setminus\{x,y\}\big)\neq\emptyset.\label{e44}\end{align}\label{lem4}
\end{lemma}

\begin{proof} We now distinguish three cases.

\emph{Case 1:} $\|x-y\|\leq d_n$. There is $z\in \mathcal{B}(y,d_n)\cap \big(S_n\setminus\{x,y\}\big)$. Since $\alpha_n > 2 d_n$, then $z\in \mathcal{B}(x,\alpha_n)$ and~\eqref{e44} holds.

\emph{Case 2:} $d_n<\|x-y\|\leq 2d_n$. There is $z\in \mathcal{B}(x,d_n)\cap \big(S_n\setminus\{x,y\}\big)$. Since $\alpha_n > 2 d_n$, then $z\in \mathcal{B}(y,\alpha_n)$ and~\eqref{e44} holds.

\emph{Case 3:} $2d_n<\|x-y\| \leq \alpha_n$. Let $p = x + d_n(y-x)$. Since $\alpha_n \geq 2d_n$, then $\mathcal{B}(p,d_n)\subseteq \mathcal{B}(x,\alpha_n)\cap\mathcal{B}(y,\alpha_n)$. Since $\mathcal{B}(p,d_n)\cap S_n\neq\emptyset$ and $\{x,y\}\cap\mathcal{B}(p,d_n)=\emptyset$, so~\eqref{e44} holds.

The combination of the above three cases establishes the lemma.
\end{proof}

\subsubsection{Convergence of fixed points}

Based on~\cite{PC-MQ-PSP:99}, the following lemma shows that
there is a unique fixed point $v_n^*$ of $\FF_n$. More importantly, $v_n^*$ provides a consistent approximation of $v^*$, and
converges to $v^*$ pointwise.

\begin{lemma} Suppose that Assumption~\ref{asm1} holds. The following properties hold:
\begin{enumerate}
\item[(Q1)] For each $n\geq1$, there is a fixed point $v^*_n$ of $\FF_n$; i.e., $v^*_n = \FF_n\circ v^*_n$, where $v^*_n : S_n \rightarrow [0,1]$;
\item[(Q2)] The sequence $v^*_n$ converges to $v^*$ pointwise; i.e., for any $x\in \mathcal{X}_{\rm free}$, it holds that \begin{align*}v^*(x) = \lim_{n\rightarrow+\infty}\min_{y\in \mathcal{B}(x,d_n)\cap S_n}v^*_n(y).\end{align*}
\end{enumerate}\label{lem5}
\end{lemma}

\begin{proof} It is easy to see that $h_n\rightarrow0^+$ and $\frac{d_n}{h_n}\rightarrow0^+$. By following the same lines towards Corollary 5.6 in~\cite{PC-MQ-PSP:99}, one can show:\begin{enumerate}
\item[(R1)] For each $n\geq1$, there is a fixed point $T^*_n$ of $\GG_n$; i.e., $T^*_n = \GG_n\circ T^*_n$, where $T^*_n : S_n \rightarrow [0,1]$;
\item[(R2)] The sequence $T^*_n$ converges to $T^*$ pointwise; i.e., for any $x\in \mathcal{X}_{\rm free}$, it holds that \begin{align*}T^*(x) = \lim_{n\rightarrow+\infty}\min_{y\in \mathcal{B}(x,d_n)\cap S_n}T^*_n(y).\end{align*}
\end{enumerate}

Given the fixed point $T^*_n$, let $v^*_n = \Psi\circ T^*_n$. The property (Q1) holds via verifying the following: \begin{align*}&\FF_n\circ v^*_n = \FF_n\circ\Psi\circ T^*_n = \FF_n\circ\Psi\circ \Psi^{-1}\circ v^*_n\\
&= \Psi\circ\GG_n\circ\Psi^{-1}\circ v^*_n = \Psi\circ\GG_n\circ T^*_n = \Psi\circ T^*_n = v^*_n,\end{align*} where in the third equality we use $\FF_n = \Psi\circ\GG_n\circ\Psi^{-1}$, and in the fifth equality we use the fixed point property of $\GG_n\circ T^*_n = T^*_n$. One can easily verify (Q2) from (R2) by using the relation $v = \Psi\circ T$.
\end{proof}

For the discrete fixed point $v^*_{n-1}$, we define the interpolated fixed point $\tilde{v}^*_{n-1} : S_n \rightarrow
[0,1]$ as follows: \begin{align*}\tilde{v}^*_{n-1}(x) = v^*_{n-1}(x),\quad x\in S_{n-1},\quad \tilde{v}^*_{n-1}(y_n) =
1.\end{align*} The following lemma shows the interpolation of fixed points is consistent.

\begin{lemma} for any $x\in \mathcal{X}$, it holds that:\begin{align}\min_{y\in\mathcal{B}(x,\alpha_n)\cap S_n}\tilde{v}_{n-1}^*(y) = \min_{y\in\mathcal{B}(x,\alpha_n)\cap S_{n-1}}v_{n-1}^*(y).\label{e35}\end{align}\label{lem6}\end{lemma}

\begin{proof} When $y_n\notin\mathcal{B}(x,\alpha_n)$, then~\eqref{e35} holds. Now consider the case of $y_n\in\mathcal{B}(x,\alpha_n)$. By Lemma~\ref{lem4} with $y = y_n$, we have $\mathcal{B}(x,\alpha_n)\cap\mathcal{B}(y_n,\alpha_n)\cap S_{n-1}\neq\emptyset$. Recall $\tilde{v}_{n-1}(y_n) = 1 \geq \max_{x\in \mathcal{B}(y_n,\alpha_n)\cap S_{n-1}}v_{n-1}(x)$. This implies that \begin{align}\tilde{v}_{n-1}(y_n)\geq \min_{y\in\mathcal{B}(x,\alpha_n)\cap S_n}\tilde{v}_{n-1}^*(y),\nnum\end{align} and thus \begin{align}\tilde{v}_{n-1}(y_n)&\geq \min_{y\in\mathcal{B}(x,\alpha_n)\cap S_{n-1}}\tilde{v}_{n-1}^*(y)\nnum\\
&= \min_{y\in\mathcal{B}(x,\alpha_n)\cap S_{n-1}}v_{n-1}^*(y).\label{e42}\end{align} The relation~\eqref{e42} implies~\eqref{e35}.
\end{proof}

In Lemma~\ref{lem5}, we show the pointwise convergence of $v_n^*$ to $v^*$. The following lemma then shows that such
convergence holds for the whole set $\mathcal{X}$.

%\begin{lemma} The following holds with probability one: \begin{align*}\|\min_{y\in\mathcal{B}(x,\alpha_n)\cap S_n}\tilde{v}_{n-1}^*(y) - \min_{z\in\mathcal{B}(x,\alpha_n)\cap S_n}v_n^*(z)\|_X\rightarrow0.\end{align*}\end{lemma}

\subsection{Proof for Lemma~\ref{lem3}:}

By using the above intermediate lemmas, we now proceed to show Lemma~\ref{lem3}.

\begin{proof} Since~\eqref{e35} holds for any $x\in \mathcal{X}$, then we have \begin{align}\|\min_{y\in\mathcal{B}(x,\alpha_n)\cap S_n}\tilde{v}_{n-1}^*(y) - \min_{y\in\mathcal{B}(x,\alpha_n)\cap S_{n-1}}v_{n-1}^*(y)\|_{\mathcal{X}} = 0.\nnum\end{align} With this, we have \begin{align}&\|\min_{y\in\mathcal{B}(x,\alpha_n)\cap S_n}\tilde{v}_{n-1}^*(y) - \min_{z\in\mathcal{B}(x,\alpha_n)\cap S_n}v_n^*(z)\|_{\mathcal{X}}\nnum\\
&= \|\min_{y\in\mathcal{B}(x,\alpha_n)\cap S_{n-1}}v_{n-1}^*(y) - \min_{z\in\mathcal{B}(x,\alpha_n)\cap S_n}v_n^*(z)\|_{\mathcal{X}}.\label{e45}\end{align}
By Lemma~\ref{lem4}, we know that $\alpha_n \geq d_{n-1}$. Then it follows from~\eqref{e45} and (T2) Lemma~\ref{lem1} that \begin{align}&\|\min_{y\in\mathcal{B}(x,\alpha_n)\cap S_n}\tilde{v}_{n-1}^*(y) - \min_{z\in\mathcal{B}(x,\alpha_n)\cap S_n}v_n^*(z)\|_{\mathcal{X}}\nnum\\
&\leq \|\min_{y\in\mathcal{B}(x,d_{n-1})\cap S_{n-1}}v_{n-1}^*(y) - \min_{z\in\mathcal{B}(x,d_n)\cap S_n}v_n^*(z)\|_{\mathcal{X}}.\label{e49}\end{align}

Combine~\eqref{e49}, (Q2) in Lemmas~\ref{lem5}, and we reach the desired result.
\end{proof}

\subsection{Proof for Theorem~\ref{the1}}

We are now in the position to show Theorem~\ref{the1}.

\begin{proof} We first show (P1). By Theorem~\ref{the2}, we have: \begin{align}&\|v_n - v^*_n\|_{K_n} = \|F_n\circ \tilde{v}_{n-1} - F_n\circ v^*_n\|_{K_n}\nnum\\
&\leq \|F_n\circ \tilde{v}_{n-1} - F_n\circ v^*_n\|_{S_n}\nnum\\
&\leq e^{-\kappa_n}\|\min_{y\in\mathcal{B}(x,\alpha_n)\cap S_n}\tilde{v}_{n-1}(y) - \min_{z\in\mathcal{B}(x,\alpha_n)\cap S_n}v_n^*(z)\|_{\mathcal{X}}\nnum\\
&\leq e^{-\kappa_n}\big(\|\min_{y\in\mathcal{B}(x,\alpha_n)\cap S_n}\tilde{v}_{n-1}(y)\nnum\\
&- \min_{z\in\mathcal{B}(x,\alpha_n)\cap S_n}\tilde{v}_{n-1}^*(z)\|_{\mathcal{X}}\nnum\\
&+\|\min_{y\in\mathcal{B}(x,\alpha_n)\cap S_n}\tilde{v}_{n-1}^*(y) - \min_{z\in\mathcal{B}(x,\alpha_n)\cap S_n}v_n^*(z)\|_{\mathcal{X}}\big).\label{e11}\end{align}

Consider the first term on the right-hand side of~\eqref{e11}. We have the following estimates: \begin{align}&\|\min_{y\in\mathcal{B}(x,\alpha_n)\cap S_n}\tilde{v}_{n-1}(y) - \min_{z\in\mathcal{B}(x,\alpha_n)\cap S_n}\tilde{v}_{n-1}^*(z)\|_{\mathcal{X}}\nnum\\
&\leq \|\tilde{v}_{n-1} - \tilde{v}_{n-1}^*\|_{S_n}\nnum\\
&= \max\{\|v_{n-1} - v_{n-1}^*\|_{S_{n-1}}, \|\tilde{v}_{n-1}(y_n) - \tilde{v}_{n-1}^*(y_n)\|\}\nnum\\
&= \|v_{n-1} - v_{n-1}^*\|_{S_{n-1}},\label{e15}\end{align} where the last inequality is an application of~\eqref{e46}.

Substitute the relation of~\eqref{e15} into~\eqref{e11}, and it renders that \begin{align}&\|v_n - v^*_n\|_{K_n}\leq e^{-\kappa_n}\big(\|v_{n-1} - v^*_{n-1}\|_{S_{n-1}}\nnum\\
&+\|\min_{y\in\mathcal{B}(x,\alpha_n)\cap S_n}\tilde{v}_{n-1}^*(y) - \min_{z\in\mathcal{B}(x,\alpha_n)\cap S_n}v_n^*(z)\|_{\mathcal{X}}\big).\label{e6}\end{align}

For each $x\in S_n\setminus\big(K_n\cup\mathcal{B}(\mathcal{X}_{\rm goal},M h_n + d_n)\big)$, we then have the
 following: \begin{align}&\|v_n(x) - v^*_n(x)\|\nnum\\
&= \|\min_{y\in\mathcal{B}(x,\alpha_{n-1})\cap S_{n-1}}v_{n-1}(y) - \min_{z\in\mathcal{B}(x,\alpha_n)\cap S_n}v^*_n(z)\|\nnum\\
&\leq \|\min_{y\in\mathcal{B}(x,\alpha_{n-1})\cap S_{n-1}}v_{n-1}(y)\nnum\\
&- \min_{z\in\mathcal{B}(x,\alpha_{n-1})\cap S_{n-1}}v^*_{n-1}(z)\|\nnum\\
&+ \|\min_{y\in\mathcal{B}(x,\alpha_{n-1})\cap S_{n-1}}v^*_{n-1}(y) - \min_{z\in\mathcal{B}(x,\alpha_n)\cap S_n}v^*_n(z)\|\nnum\\
&\leq \|v_{n-1} - v^*_{n-1}\|_{S_{n-1}}\nnum\\
&+ \|\min_{y\in\mathcal{B}(x,\alpha_{n-1})\cap S_{n-1}}v^*_{n-1}(y) - \min_{z\in\mathcal{B}(x,\alpha_n)\cap S_n}v^*_n(z)\|_{\mathcal{X}}\nnum\\
&= \|v_{n-1} - v^*_{n-1}\|_{S_{n-1}}\nnum\\
&+ \|\min_{y\in\mathcal{B}(x,\alpha_n)\cap S_n}\tilde{v}^*_{n-1}(y) - \min_{z\in\mathcal{B}(x,\alpha_n)\cap S_n}v^*_n(z)\|_{\mathcal{X}},\label{e39}\end{align} where we apply Lemma~\ref{lem1} in the second inequality and Lemma~\ref{lem6} in the last inequality.

For each $x\in S_n\cap\mathcal{B}(\mathcal{X}_{\rm goal},M h_n + d_n)$, we have \begin{align}v_n(x) = v^*_n(x).\label{e40}\end{align}

Recall Assumption~\ref{asm2} that $S_n\subseteq\cup_{s=0}^DK_{n+s}$. With this property, we combine~\eqref{e6},~\eqref{e39} and~\eqref{e40}, and reach the following: \begin{align}&\|v_n(x) - v^*_n(x)\|_{S_n}\nnum\\
&\leq e^{-\kappa_n}\big(\|v_{n-D-1} - v^*_{n-D-1}\|_{S_{n-D-1}}+\gamma_{n-1}\big),\label{e47}\end{align} where the sequence $\{\gamma_n\}$ is defined in~\eqref{e17}.

Denote $\beta_n \triangleq \|v_n - v^*_n\|_{S_n}$. It follows from~\eqref{e47} that\begin{align}\beta_{n+1}
\leq e^{-\kappa_{n+1}}\big(\beta_{n-D}+\gamma_n\big).\label{e29}\end{align} Denote by $\xi_n \triangleq \frac{\beta_n}{\gamma_n^{1-\epsilon}}$. The following is derived from~\eqref{e29} for $n\geq D$: \begin{align}\xi_{n+1} &= \frac{\beta_{n+1}}{\gamma_{n+1}^{1-\epsilon}}\leq
\big(\frac{\gamma_{n-D}}{\gamma_{n+1}}\big)^{1-\epsilon}
e^{-\kappa_{n+1}}\big(\frac{\beta_{n-D}}{\gamma_{n-D}^{1-\epsilon}}+\frac{\gamma_n}{\gamma_{n-D}^{1-\epsilon}}\big)\nnum\\
&\leq C^{(D+2)(1-\epsilon)}e^{-\kappa_{n+1}}\big(\xi_{n-D}+\gamma_{n-D}^{\epsilon}\big).\label{e30}
\end{align}

In Lemma~\ref{lem2}, it is shown that $d_s(\frac{1}{n})^{\frac{1}{N}} \leq d_n \leq D_s(\frac{\log n}{n})^{\frac{1}{N}}$ for $n\geq \max\{n_s,N_s\}$. Since $\frac{d_n}{h_n}\rightarrow0$, there is $n_1\geq \max\{n_s,N_s\}$ such that the following holds for $n\geq n_1$: \begin{align}\kappa_n &= h_n - d_n = h_n(1-\frac{d_n}{h_n})\nnum\\
&\geq \frac{h_n}{2}\geq \eta(\frac{1}{n})^{\frac{1}{N(1+\alpha)}},\label{e23}\end{align} where $\eta\triangleq \frac{1}{2}(d_s)^{\frac{1}{1+\alpha}}$. Then, it follows from~\eqref{e30} and~\eqref{e23} that for $n\geq n_1$: \begin{align}\xi_{n+1}&\leq C^{(D+2)(1-\epsilon)}e^{-\eta(\frac{1}{n+1})^{\frac{1}{N(1+\alpha)}}}(\xi_{n-D} + \gamma_{n-D}^{\epsilon})\nnum\\
&\leq C^{(D+2)(1-\epsilon)}e^{-\eta\sum_{s=n_1}^{n+1}(\frac{1}{s})^{\frac{1}{N(1+\alpha)}}}\xi_1\nnum\\ &+C^{(D+2)(1-\epsilon)}\sum_{\tau=n_1}^{n+1}e^{-\eta\sum_{s=\tau}^{n+1}(\frac{1}{s})^{\frac{1}{N(1+\alpha)}}}
\gamma_{\tau}^{\epsilon}.\label{e12}\end{align}

We now proceed to show that $\xi_n\rightarrow0$ via~\eqref{e12}. Let $\delta_n\triangleq\gamma_n^{\epsilon}$. Since $\frac{1}{N(1+\alpha)} \in (0,1)$, the harmonic sequence of $\{(\frac{1}{s})^{\frac{1}{N(1+\alpha)}}\}_{s\geq1}$ is not summable. So $e^{-\eta\sum_{s=n_1}^{n+1}s^{-\frac{1}{N(1+\alpha)}}}\rightarrow0$ as $n\rightarrow+\infty$. We now consider the term of $\sum_{\tau=n_1}^{n+1}e^{-\eta\sum_{s=\tau}^{n+1}(\frac{1}{s})^{\frac{1}{N(1+\alpha)}}}
\delta_{\tau}$. Recall that $\{\delta_{\tau}\}$ diminishes. Pick any $\varepsilon>0$. There is $n_2 \geq n_1$ such that $\|\delta_{\tau}\|\leq\epsilon$ for all $n\geq n_2$. So we have \begin{align}&\sum_{\tau=n_1}^{n+1}e^{-\eta\sum_{s=\tau}^{n+1}s^{-\frac{1}{N(1+\alpha)}}}
\|\delta_{\tau}\|\nnum\\
&\leq \sum_{\tau=n_1}^{n_2}e^{-\eta\sum_{s=\tau}^{n+1}s^{-\frac{1}{N(1+\alpha)}}}
\|\delta_{\tau}\|\nnum\\
&+ \varepsilon\sum_{\tau=n_2+1}^{n+1}e^{-\eta\sum_{s=\tau}^{n+1}s^{-\frac{1}{N(1+\alpha)}}}.\label{e27}\end{align} Since $\{\delta_{\tau}\}$ converges, so is uniformly bounded. Then we have the following for the first term on the right-hand side of~\eqref{e27} by taking $n\rightarrow+\infty$: \begin{align}\sum_{\tau=n_1}^{n_2}e^{-\eta\sum_{s=\tau}^{n+1}s^{-\frac{1}{N(1+\alpha)}}}
\|\delta_{\tau}\| \rightarrow 0.\label{e13}\end{align} We now consider the second term on the right-hand side of~\eqref{e27}. Since $\{(\frac{1}{s})^{\frac{1}{N(1+\alpha)}}\}$ is monotonically decreasing, then we have \begin{align}&\sum_{s=\tau}^{n+1}s^{-\frac{1}{N(1+\alpha)}}
\geq\int_{\tau}^{n+1}x^{-\frac{1}{N(1+\alpha)}}dx\nnum\\
&= \frac{1}{1-\frac{1}{N(1+\alpha)}}\big((n+1)^{1-\frac{1}{N(1+\alpha)}}
-\tau^{1-\frac{1}{N(1+\alpha)}}\big)\nnum\\
&= \frac{\tau^{1-\frac{1}{N(1+\alpha)}}}{1-\frac{1}{N(1+\alpha)}}
\big((1+\frac{n+1-\tau}{\tau})^{1-\frac{1}{N(1+\alpha)}}
-1\big)\nnum\\
&=\frac{1}{1-\frac{1}{N(1+\alpha)}}(n+1-\tau)^{1-\frac{1}{N(1+\alpha)}}.\label{e56}\end{align}
By~\eqref{e56}, we have
\begin{align}&\sum_{\tau=n_2+1}^{n+1}e^{-\eta\sum_{s=\tau}^{n+1}s^{-\frac{1}{N(1+\alpha)}}}\nnum\\
&\leq\sum_{\tau=n_2+1}^{n+1}
e^{-\frac{\eta}{1-\frac{1}{N(1+\alpha)}}(n+1-\tau)^{1-\frac{1}{N(1+\alpha)}}}\nnum\\
&=\sum_{\tau=0}^{n-n_2}
e^{-\frac{\eta}{1-\frac{1}{N(1+\alpha)}}\tau^{1-\frac{1}{N(1+\alpha)}}}\nnum\\
&\leq\sum_{\tau=0}^{+\infty}
e^{-\frac{\eta}{1-\frac{1}{N(1+\alpha)}}\tau^{1-\frac{1}{N(1+\alpha)}}}<+\infty.\label{e14}\end{align}

%Recall the Bernoulli's inequality $(1+x)^r \geq 1+rx$. So it follows from~\eqref{e56} that \begin{align}&\sum_{s=\tau}^{n+1}s^{-\frac{1}{N(1+\alpha)}}\nnum\\
%&\geq \frac{\tau^{1-\frac{1}{N(1+\alpha)}}}{1-\frac{1}{N(1+\alpha)}}
%\big((1+\frac{n+1-\tau}{\tau}({1-\frac{1}{N(1+\alpha)}}))
%-1\big)\nnum\\
%&=\frac{n+1-\tau}{\tau^{\frac{1}{N(1+\alpha)}}}.\label{e57}\end{align}
%
%By~\eqref{e57}, we have
%\begin{align}&\sum_{\tau=n_2+1}^{n+1}e^{-\eta\sum_{s=\tau}^{n+1}s^{-\frac{1}{N(1+\alpha)}}}\nnum\\
%&\leq\sum_{\tau=n_2+1}^{n+1}
%e^{-\eta(n+1-\tau)\tau^{-\frac{1}{N(1+\alpha)}}}<+\infty.\label{e14}\end{align}

Substitute~\eqref{e13} and~\eqref{e14} into~\eqref{e12}, and we have the following for all $n\geq1$: \begin{align}\xi_{n+1}\leq O(\varepsilon).\nnum\end{align} Take the limit on both sides of the above relation, and it renders that \begin{align}\limsup_{n\rightarrow+\infty}\xi_n\leq O(\varepsilon),\nnum\end{align} and then \begin{align}\lim_{n\rightarrow+\infty}\xi_n = O(\varepsilon).\label{e32}\end{align} Since~\eqref{e32} holds for any $\varepsilon >0$, we conclude that $\xi_n = \frac{\beta_n}{\gamma_n^{1-\epsilon}} \rightarrow0$ and then the second property of (P1) holds. This further implies that \begin{align}&\frac{1}{\gamma_n^{1-\epsilon}}\|\min_{y\in \mathcal{B}(x,\alpha_n)\cap S_n}v_n(y)
- \min_{z\in \mathcal{B}(x,\alpha_n)\cap S_n}v_n^*(z)\|_{\mathcal{X}}\nnum\\
&\leq \frac{\|v_n-v_n^*\|_{S_n}}{\gamma_n^{1-\epsilon}} \rightarrow 0,\nnum\end{align} where Lemma~\ref{lem2} is invoked. We then establish the second property of (P1). The property (P2) is a direct result of the combination of (Q2) in Lemma~\ref{lem5} and (P1).
\end{proof}

\end{document}